\documentclass[11pt]{amsart}

\usepackage{mathtools,amsmath}
\usepackage{amsxtra}
\usepackage{amscd}
\usepackage{amsthm}
\usepackage{amsfonts}
\usepackage{amssymb}
\usepackage{eucal}
\usepackage{color}
\usepackage[all]{xy}
\usepackage{enumerate}
\usepackage{mathabx}

\newcommand{\nc}{\newcommand}
\newcommand{\rc}{\renewcommand}

\numberwithin{equation}{section}

\newtheorem{thm}{Theorem} [section]
\newtheorem{prop}[thm]{Proposition}
\newtheorem{lem}[thm]{Lemma}

\newtheorem{rmk}[thm]{Remark}

\newtheorem{important note}[thm]{Important Note}

\newtheorem*{thm*}{Theorem}
\nc{\on}{\operatorname}

\nc{\Lemma}{\begin{lem}}
\nc{\enlemma}{\end{lem}}
\nc{\Proof}{\begin{proof}}

\nc{\bA}{{\mathbb A}}
\nc{\bB}{{\mathbb B}}
\nc{\bC}{{\mathbb C}}
\nc{\bD}{{\mathbb D}}
\nc{\bE}{{\mathbb E}}
\nc{\bF}{{\mathbb F}}
\nc{\bG}{{\mathbb G}}
\nc{\bH}{{\mathbb H}}
\nc{\bI}{{\mathbb I}}
\nc{\bJ}{{\mathbb J}}
\nc{\bK}{{\mathbb K}}
\nc{\bL}{{\mathbb L}}
\nc{\bM}{{\mathbb M}}
\nc{\bN}{{\mathbb N}}
\nc{\bO}{{\mathbb O}}
\nc{\bP}{{\mathbb P}}
\nc{\bQ}{{\mathbb Q}}
\nc{\bR}{{\mathbb R}}
\nc{\bS}{{\mathbb S}}
\nc{\bT}{{\mathbb T}}
\nc{\bU}{{\mathbb U}}
\nc{\bV}{{\mathbb V}}
\nc{\bW}{{\mathbb W}}
\nc{\bZ}{{\mathbb Z}}
\nc{\bX}{{\mathbb X}}
\nc{\bY}{{\mathbb Y}}

\def\bbC{{\mathbb C}}

\def\bbN{{\mathbb N}}

\def\bbR{{\mathbb R}}

\def\bbZ{{\mathbb Z}}

\nc{\cA}{{\mathcal A}}
\nc{\cB}{{\mathcal B}}
\nc{\cC}{{\mathcal C}}
\nc{\cD}{{\mathcal D}}
\nc{\cE}{{\mathcal E}}
\nc{\cF}{{\mathcal F}}
\nc{\cG}{{\mathcal G}}
\nc{\cH}{{\mathcal H}}
\nc{\cI}{{\mathcal I}}
\nc{\cJ}{{\mathcal J}}
\nc{\cK}{{\mathcal K}}
\nc{\cL}{{\mathcal L}}
\nc{\cM}{{\mathcal M}}
\nc{\cN}{{\mathcal N}}
\nc{\cO}{{\mathcal O}}
\nc{\cP}{{\mathcal P}}
\nc{\cQ}{{\mathcal Q}}
\nc{\cR}{{\mathcal R}}
\nc{\cS}{{\mathcal S}}
\nc{\cT}{{\mathcal T}}
\nc{\cU}{{\mathcal U}}
\nc{\cV}{{\mathcal V}}
\nc{\cW}{{\mathcal W}}
\nc{\cZ}{{\mathcal Z}}
\nc{\cX}{{\mathcal X}}
\nc{\cY}{{\mathcal Y}}

\nc{\fA}{{\mathfrak A}}
\nc{\fB}{{\mathfrak B}}
\nc{\fC}{{\mathfrak C}}
\nc{\fD}{{\mathfrak D}}
\nc{\fE}{{\mathfrak E}}
\nc{\fF}{{\mathfrak F}}
\nc{\fG}{{\mathfrak G}}
\nc{\fH}{{\mathfrak H}}
\nc{\fI}{{\mathfrak I}}
\nc{\fJ}{{\mathfrak J}}
\nc{\fK}{{\mathfrak K}}
\nc{\fL}{{\mathfrak L}}
\nc{\fM}{{\mathfrak M}}
\nc{\fN}{{\mathfrak N}}
\nc{\fO}{{\mathfrak O}}
\nc{\fP}{{\mathfrak P}}
\nc{\fQ}{{\mathfrak Q}}
\nc{\fR}{{\mathfrak R}}
\nc{\fS}{{\mathfrak S}}
\nc{\fT}{{\mathfrak T}}
\nc{\fU}{{\mathfrak U}}
\nc{\fV}{{\mathfrak V}}
\nc{\fW}{{\mathfrak W}}
\nc{\fZ}{{\mathfrak Z}}
\nc{\fX}{{\mathfrak X}}
\nc{\fY}{{\mathfrak Y}}
\nc{\fa}{{\mathfrak a}}
\nc{\fb}{{\mathfrak b}}
\nc{\fc}{{\mathfrak c}}
\nc{\fd}{{\mathfrak d}}
\nc{\fe}{{\mathfrak e}}
\nc{\ff}{{\mathfrak f}}
\nc{\fg}{{\mathfrak g}}
\nc{\fh}{{\mathfrak h}}
\nc{\fiI}{{\mathfrak i}}  
\nc{\ffi}{{\mathfrak i}}  
\nc{\fj}{{\mathfrak j}}
\nc{\fk}{{\mathfrak k}}
\nc{\fl}{{\mathfrak{l}}}
\nc{\fm}{{\mathfrak m}}
\nc{\fn}{{\mathfrak n}}
\nc{\fo}{{\mathfrak o}}
\nc{\fp}{{\mathfrak p}}
\nc{\fq}{{\mathfrak q}}
\nc{\fr}{{\mathfrak r}}
\nc{\fs}{{\mathfrak s}}
\nc{\ft}{{\mathfrak t}}
\nc{\fu}{{\mathfrak u}}
\nc{\fv}{{\mathfrak v}}
\nc{\fw}{{\mathfrak w}}
\nc{\fz}{{\mathfrak z}}
\nc{\fx}{{\mathfrak x}}
\nc{\fy}{{\mathfrak y}}

\nc{\al}{{\alpha }}
\nc{\ga}{{\gamma }}
\nc{\de}{{\delta }}
\nc{\del}{{\partial }}
\nc{\ep}{{\varepsilon }}
\nc{\vap}{{\epsilon }}

\nc{\ze}{{\zeta }}
\nc{\et}{{\eta }}
\rc{\th}{{\theta }}

\nc{\vth}{{\vartheta }}

\nc{\io}{{\iota }}
\nc{\ka}{{\kappa }}
\nc{\la}{{\lambda }}
\nc{\vrho}{{\varrho}}
\nc{\si}{{\sigma }}
\nc{\ups}{{\upsilon }}
\nc{\vphi}{{\varphi }}
\nc{\om}{{\omega }}

\nc{\Ga}{{\Gamma }}
\nc{\De}{{\Delta }}
\nc{\nab}{{\nabla}}
\nc{\Th}{{\Theta }}
\nc{\La}{{\Lambda }}
\nc{\Si}{{\Sigma }}
\nc{\Ups}{{\Upsilon }}
\nc{\Om}{{\Omega }}

\nc{\hE}{\hat{\cE}}
\nc{\hA}{\widehat{A}}
\nc{\hK}{\widehat{K}}
\nc{\hM}{\hat{\cM}}
\nc{\hN}{\hat{\cN}}
\nc{\hF}{\hat{\cF}}
\nc{\hcA}{\widehat\cA}
\nc{\hcK}{\widehat\cK}
\nc{\tN}{\widetilde\cN}
\nc{\tM}{\widetilde\cM}

\nc{\Coh}{{{\mathcal C}oh}}
\nc{\Loc}{{{\mathcal L}oc}}
\nc{\GR}{{G_\bR}}

 \providecommand{\curlybrackets}[1]{{\{\!\{#1\}\!\}}}
 \nc{\cb}{\curlybrackets}

\nc{\str}{{\hcA}}

\nc{\Spec}{{\on{Spec}}}
\nc{\cext}{{\cE\mathit x\mathit t}}
\nc{\ctor}{{\cT\mathit o\mathit r}}
\nc{\crhom}{{\operatorname{R}\cH\mathit o\mathit m}}
\nc{\chom}{{\cH\mathit o\mathit m}}
\nc{\oh}{{\on{H}}}
\nc{\codim}{{\on{codim}}}
\nc{\Supp}{{\on{Supp}}}
\nc{\coker}{{\on{coker}}}
\nc{\id}{\mathrm{id}}
\nc{\seteq}{\mathbin{:=}}

\nc{\dA}[1]{\operatorname{D_\cA}(#1)}
\nc{\cl}{\colon}
\nc{\Ker}{\on{ker}}
\nc{\wcheck}{\widecheck}

\nc{\rhD}{{\operatorname{Mod}_{hr}(\cD_X)}}
\nc{\rhDL}{{\operatorname{Mod}_{hr}(\cD_X)_\La}}
\nc{\rhE}{{\operatorname{Mod}_{hr}(\cE_X)}}
\nc{\rhEL}{{\operatorname{Mod}_{hr}(\cE_X)_\La}}
\nc{\rhFE}{{\operatorname{Mod}_{hr}(\hat\cE_X)}}
\nc{\rhFEL}{{\operatorname{Mod}_{hr}(\hat\cE_X)_\La}}
\nc{\op}{{\on{P}}}
\nc{\oM}{{\on{M}}}
\nc{\ba}{\begin{array}}
\nc{\ea}{\end{array}}
\nc{\hs}{\hspace*}

\newcommand{\scbul}{{\,\raise.4ex\hbox{$\scriptscriptstyle\bullet$}\,}}
\nc{\dT}{{\mathring{T}}{}^*}
\nc{\oY}{{\mathring{Y}}}
\nc{\oLa}{{\mathring{\La}}}
\nc{\eq}{\begin{eqnarray}}
\nc{\eneq}{\end{eqnarray}}
\nc{\eqn}{\begin{eqnarray*}}
\nc{\eneqn}{\end{eqnarray*}}
\nc{\Per}{\on{\mathcal{P}\mathit{er}}}
\rc{\setminus}{{-}}

\nc{\bigmid}{\;\mathbin{\rule[-1.8ex]{.5pt}{3.8ex}}\;}

\nc{\od}{{\operatorname{D}}}
\nc{\odh}{{\operatorname{D}_{\hcA}}}
\nc{\gl}{{\mathfrak{gl}}}

\nc{\Z}{\bbZ}
\nc{\C}{\bC}

\nc{\be}{\begin{enumerate}}
\nc{\ee}{\end{enumerate}}
\nc{\bnum}{\be[{\rm(i)}]}
\nc{\enum}{\ee}
\nc{\Mod}{\on{Mod}}
\nc{\coh}{\mathrm{coh}}
\nc{\sing}{\mathrm{sing}}
\nc{\drho}{{\mathring{\rho}}}
\nc{\bbL}{\mathbb{L}at}
\nc{\bl}{\bigl}
\nc{\br}{\bigr}

\def\bea{\begin{eqnarray}}
\def\eea{\end{eqnarray}}

\def\beq{\begin{equation}}
\def\eeq{\end{equation}}

\def\beqn{\begin{equation*}}
\def\eeqn{\end{equation*}}

\def\beal{\begin{align*}}

\def\eeal{ \end{align*} }

\def\BET{\begin{theorem}}
\def\ENT{\end{theorem}}
\def\BEP{\begin{proposition}}
\def\ENP{\end{proposition}}
\def\BEL{\begin{lemma}}
\def\ENL{\end{lemma}}
\def\BEC{\begin{corollary}}
\def\ENC{\end{corollary}}
\def\BEE{\begin{example} \rm}
\def\ENE{\end{example}}
\def\BER{\begin{remark} \rm}
\def\ENR{\end{remark}}
\def\BED{\begin{definition} \rm}
\def\END{\end{definition}}
\def\BECJ{\begin{conjecture}}
\def\ENCJ{\end{conjecture}}

\nc{\ber}{\color{red}}
\nc{\er}{\color{black}}

\def\roweq{\nonumber \\ &=& }
\def\rowleq{\nonumber \\  & \leq & }
\def\rowgeq{\nonumber \\ & \geq & }

\def\BEL{\begin{lem}}
\def\ENL{\end{lem}}

\renewcommand{\Re}{\operatorname{Re}}
\renewcommand{\Im}{\operatorname{Im}}

\nc{\cmtk}[1]{\color{cyan}{{\fbox{K}} #1}\color{black}}
\nc{\cmtj}[1]{\color{red}{{\fbox{J}} #1}\color{black}}

\begin{document}

\author{Jari Taskinen}

\address{Department of Mathematics and Statistics, University of Helsinki, P.O.Box 68, FI-00014 Helsinki, Finland}
\email{jari.taskinen@helsinki.fi}
\thanks{Jari Taskinen was supported in part by the Academy of Finland and the V\"ais\"al\"a Foundation}

\author{Kari Vilonen}
 \thanks{Kari Vilonen was supported in part by NSF grants DMS-1402928 \& DMS-1069316, the Academy of Finland,  the ARC grant DP150103525, the Humboldt Foundation, and the Simons Foundation.}

\address{School of Mathematics and Statistics, University of Melbourne, VIC 3010, Australia, and Department of Mathematics and Statistics, University of Helsinki, P.O.Box 68, FI-00014 Helsinki, Finland}
         \email{kari.vilonen@unimelb.edu.au and
 kari.vilonen@helsinki.fi}

\title{Cartan Theorems for Stein manifolds over a discrete valuation base}

\maketitle

\section{Introduction}

In this paper we prove Cartan theorems A and B for Stein manifolds in a relative setting. We work over a base which is a topological discrete valuation ring satisfying certain conditions which we discuss in more detail below. 

We were led to this study by questions that arouse in the work of the second author with Kashiwara in the proof of the codimension-three conjecture for holonomic micro differential systems~\cite{KaVi}. We could obtain a natural proof of the main results of~\cite{KaVi} if we had  in our disposal a relative theory of several complex variables where the base is a (certain) topological DVR. In particular, conjecture 1.8 in \cite{KaVi} -- which states that in a relative setting reflexive coherent sheaves extend uniquely across loci of codimension at least  three -- is a direct analogue of classical results of Trautmann, Siu, and Frisch-Guenot~\cite{T,Siu,FG}.

Spaces of holomorphic functions with values in a topological vector space have been considered at least since Grothendieck~\cite{Gr}. Extending Cartan theorems A and B to the context of holomorphic functions with values on a locally convex topological vector space was considered in the papers of Bungart~\cite{Bu1}, and later also by Leiterer. However, they only consider coherent sheaves that come by extension of scalars from ordinary coherent sheaves.

In this paper we consider a different situation. Let $X$ be a complex manifold and $A$ a topological discrete valuation ring.  We write $\cA_X$ for the sheaf of functions on $X$ with values in $A$. Our goal is to prove Cartan theorems A and B for coherent $\cA_X$-modules when $X$ is a Stein manifold. For our methods to work we impose  technical conditions on $A$ which are formulated  in section~\ref{Our set up}. In particular $A$ will be a subring of the formal power series ring $\bC[[t]]$ satisfying specific ``convergence" conditions. A fundamental example of such convergence conditions, coming up in the context of microdifferential operators,  is given in~\eqref{C norm}. We topologize $A$ as a direct limit of Banach algebras, $A=\varinjlim_{h>0} A_h$, and one of our basic assumptions is that $A$ is dual nuclear Fr\'echet. 
For our arguments to work we have to control the nuclearity of $A$ a bit, as specified in condition~\eqref{1g}. Finally, condition~\eqref{3j}, which we call subharmonicity,  guarantees that we have enough analogues of pseudoconvex domains.  In a technical sense it gives us enough plurisubharmonic functions to carry out the $L^2$-analysis of H\"ormander in our context. It seems reasonable to expect that one obtains a good theory if $A$ is regular local ring satisfying conditions analogous to the ones we pose on discrete valuation rings, but we have not looked into this. In section~\ref{general results} we give two simple examples which show that Cartan theorems fail when $A$ is not local.

Our main result, (Theorem~\ref{main theorem}) is:
\begin{thm*} 
Let $A$ be a discrete valuation ring satisfying conditions~\eqref{cond} and let $X$ be a Stein manifold. If $\cF$ is a coherent  $\cA_X$-module then $\oh^i(X,\cF)=0$ for $i\geq 1$. Furthermore, the sheaf $\cF$ is generated by its global sections. 
\end{thm*}

Our proof of this theorem follows the standard strategy of first proving it for compact Stein domains and then extending it to non-compact ones by a limiting process. However, to carry out this process in our context we have to work in a more general setting where there ring $A$ varies along the complex manifold $X$. This set up is studied in section~\ref{varying levels}. 

The paper is organized as follows. In section~\ref{general results} we recall some general old results of Bungart and also give two well-known examples which show that one should not expect the Cartan theorems to hold in great generality. In section~\ref{Our set up} we formulate precise conditions on discrete valuation rings which we will be working with. We also show that standard results of several complex variables hold in our setting, in particular, that our structure sheaf is coherent. 

In section~\ref{varying levels} we introduce the notions of holomorphic functions with values in a varying topological vector space. This is an important technical tool in the proof of the main theorem. 

In sections~\ref{compact section} and~\ref{L2 arguments} we prove the Cartan theorems for compact blocks in $\bC^N$. We do so for structure sheaves which consist of holomorphic functions where the target space varies. To accomplish this we utilize $L^2$-techniques of H\"ormander. In section~\ref{approximation} we prove approximation lemmas which are used in section~\ref{main section} where we prove our main result. 

\section{Some general results}
\label{general results}

We consider a complex manifold $X$ and the sheaf of holomorphic functions $\cA_X$ on $X$ with values in a topological ring $A$.  Let us write $\cO_X$ for the sheaf of holomorphic functions on $X$, as usual. Let us recall that the sheaf $\cO_X$ has a natural structure of sheaf of topological rings if we equip it with the topology uniform convergence on compact sets. Moreover, with this structure $\cO_X$ is a nuclear Fr\'echet sheaf, i.e., for $U\subset X$ open the $\cO_X(U)$ are nuclear Fr\'echet topological rings. Let us consider a topological ring  $A$. We define the structure sheaf $\cA_X$ as follows. For any open $U$ we set $\cA_X(U) = A\hat\otimes \cO_X(U)$; here and in the rest of the paper all topological tensor products are projective tensor products. The sections $\cA_X(U)$ can also be identified with holomorphic functions on $U$ with values in $A$ which we equip with the topology uniform convergence on compact sets. 

Cartan theorems A and B were studied by Bungart for sheaves that arise from ordinary coherent sheaves on $X$ by extending scalars. He obtained the following results:

\begin{thm}(Bungart, \cite{Bu1}[Theorem B in section 11])
\label{Bungart B} 
Let $\cF$ be a coherent analytic sheaf on a Stein space $X$. Then for every Frechet space $E$ we have
$H^q(X, E\hat\otimes\cF) = 0$ for $q>0$.
\end{thm}

\begin{thm}(Bungart, \cite{Bu1}[Theorem B* in section 17])
Let $\cF$ be a coherent analytic sheaf on the Stein space $X$ and $K$ a holomorphically convex compact subset of $X$, and $E$ is any (quasi-)complete locally convex space. Then $H^q(K,E\hat\otimes\cF) = 0$ for $q>0$.
\end{thm}
\begin{thm}(Bungart, \cite{Bu1}[Theorem A in section 17]) Let $X$ be Stein space, $\cF$ is a coherent sheaf of $\cO_X$-modules on $X$,  and $E$ is any (quasi-)complete locally convex space. Then $E\hat\otimes
\cF$ is generated by global sections.  
\end{thm}

The first theorem~\ref{Bungart B} can be proved by a straightforward extensions of scalars argument. To explain it, let us recall the following:
\begin{lem} 
\label{exactness of tensor product}
Let $F$ be a locally convex space and 
$$
0 \rightarrow E_1\xrightarrow\alpha  E_2\xrightarrow\beta E_3 \rightarrow 0
$$
a topologically exact sequence of locally convex spaces. Assume that either $F$ or $E_2$ is nuclear and that
$F$ and $E_2$ are both Frechet spaces or both DF spaces. Then the sequence 
$$
0 \rightarrow F \hat\otimes E_1  \xrightarrow {1\hat\otimes\alpha} F \hat\otimes  E_2  \xrightarrow{1\hat\otimes\beta}  F \hat\otimes E_3  \rightarrow 0
$$
is exact.
\end{lem}
Let us choose a Stein cover $\cU$ of $X$ and form the Chech complex $C^\cdot(\cU,\cF)$. As the topological vector spaces in $C^\cdot(\cU,\cF)$ are nuclear Frechet the lemma implies that tensoring with $E$ commutes with taking cohomology and thus we obtain theorem~\ref{Bungart B}. 

Our interest lies in finding a good class of topological algebras $A$ so that the Cartan theorems hold for coherent $\cA_X$-modules. We present two standard examples which illustrate that one has to exercise some caution if one is to  generalize these theorems for coherent sheaves of $\cA_X$-modules even for some rather reasonable $A$.

\subsection{Example 1}

Let us take $A=\bC[t]$ and $X=\bC$. We equip $A$ with its direct limit topology induced by finite dimensional subspaces. Note that the sheaf $\cA_\bC$ is coherent.  We will consider the following exact sequence:
\begin{equation*}
0 \to \cI_\bZ \to \cO_\bC \to \cO_\bZ \to 0
\end{equation*}
where we consider $\bZ \subset \bC$ as a sub variety. We can tensor this sequence with $A$ to obtain an exact sequence
\begin{equation*}
0 \to \cI^A_\bZ \to \cA_\bC \to \cA_\bZ \to 0
\end{equation*}
where we have written, as before, $\cA_Y$ for the sheaf of holomorphic functions with values in $Y$. Now we see that on the level of global sections the last two terms become
\begin{equation*}
\Gamma(\bC, \cA_\bC) = \{\sum_{i=0}^m f_i t^i\mid f_i\in \cO_\bC(\bC)\} \to \{(n,\sum_{i=0}^{m_n} a_{i,n}t^i)\} = \Gamma(\bC, \cA_\bZ) 
\end{equation*}
and the map is given by
\begin{equation*}
\sum_{i=0}^m f_i t^i \mapsto (n,\sum_{i=0}^{m_n} f_i(n)t^i)\,.
\end{equation*}
It is clear that the element
\begin{equation*}
 (n,t^n)
\end{equation*}
cannot come from any $\sum_{i=0}^m f_i t^i $. Thus, from the exact sequence 
\begin{equation*}
0 \to \Gamma(\bC,\cA_\bZ) \to\Gamma(\bC, \cA_\bC)  \to \Gamma(\bC, \cA_\bZ)  \to \oh^1(\bC,\cA_\bZ) \to \dots
\end{equation*}
we conclude that
\begin{equation*}
\oh^1(\bC,\cI^A_\bZ ) \neq 0\,,
\end{equation*}
although $\cI^A_\bZ $ is coherent. Thus, Cartan's theorem B fails in this case.

\subsection{Example 2}

 Let us now take as our $A$ the ring $A=\bC[t,t^{-1}]$ and for our $X$ we take $X=\bC/\bZ$. Again we equip $A$ with its direct limit topology induced by finite dimensional subspaces. We consider the constant sheaf $\cA_\bC$ on $\bC$ and consider the $\bZ$-action on it via $(n\cdot f)(t) = t^n f(t-n)$.  Viewed in this manner $\cA_\bC$ is a $\bZ$-equivariant sheaf on $\bC$ and so it induces a sheaf $\cA_\bC/\bZ$ on $X=\bC/\bZ$. This sheaf is clearly locally free, so it is coherent. We now consider
 \begin{equation*}
\Gamma(\bC/\bZ,\cA_\bC/\bZ ) = \Gamma(\bC,\cA_\bC)^\bZ \,.
\end{equation*}
Clearly,
\begin{equation*}
\Gamma(\bC,\cA_\bC)^\bZ = 0\,.
\end{equation*}
Thus, Cartan's theorem A fails as the sheaf is not generated by its global sections because there are not any.

\section{Our set up}
\label{Our set up}

In this section we explain the conditions we will be imposing on the topological ring $A$. 
We continue to consider a complex manifold $X$ and the sheaf of holomorphic functions $\cA_X$ on $X$ with values in a topological ring $A$ as was explained in the previous section. 

From the point of view of \cite{KaVi} the ring of interest is the following regular local ring A. Consider the formal power series ring $\hA = \bC[[t]]$. It is a discrete valuation ring. We define a subring $A$ of $\hA$ in the following manner. 
For any $h>0$ we define a norm $\|\ \|_h$ on $\hA$ by the formula
\begin{equation}
\label{C norm}
\|\sum_{j=0}^\infty a_jt^j\|_{_h}\ = \ \sum_{j=0}^\infty |a_j|\frac{h^j}{j!} 
\,.
\end{equation}
We write $A_h$ for the subring consisting of elements $a$ of $\hA$ with 
$\| a\| _h<\infty$. The ring $A_h$ is a Banach local ring as is not so difficult to see. 
Finally, we set 
\begin{equation*}
A = \varinjlim _{h \to 0}A_h\,.
\end{equation*}
The topological ring $A$ is a dual nuclear Fr\'echet discrete valuation ring, a DNF DVR. 

We will next define a class of topological rings that we will be working with which includes the example discussed above. 
 First of all, we need to impose reasonable conditions on $A$ so that $\cA_X$ is coherent and so that the stalks $\cA_{X,x}$ are regular local rings. Thus, we have to assume that $A$ is a regular local ring. In light of Example 1 of section~\ref{general results} this is a reasonable assumption anyway. Let us write $\fm$ for the maximal ideal in $A$ and let us assume that the dimension of $A$ is $r$. We complete $A$ with respect to $\fm$ to obtain a complete local ring $\hat A  \cong \bC[[t_1, ... , t_r]]$. Then 
\begin{equation*}
\bC[t_1, ... , t_r] \subset A \subset \hat A  \cong \bC[[t_1, ... , t_r]]\,.
\end{equation*}
Hence $A$ can be viewed as  consisting of power series in $r$ variables satisfying some kind of a ``convergence" condition. 
The main results in this paper should hold for regular local rings of any dimension. However, we will make the further assumption that $\dim A=1$, i.e., that $A$ is a DVR. As $A\subset \bC[[t]]$, any element $f\in A$ can be written as
\beq
\label{representation}
f \ = \ \sum_{j=0}^\infty a_jt^j\,.
\eeq
We can thus view the $A$ as providing us with a (usually very small!) neighborhood of the origin in $\bC$. 

We assume that the topology on $A$ is given as a direct limit of Banach algebras, i.e., that
\begin{equation*}
A\ = \ \varinjlim_{h>0} A_h
\end{equation*}
where the $A_h$, with $0<h<S$ for some fixed $S\in\bR^+$,  are commutative Banach algebras with norm $\|\  \|_h$ and we will assume furthermore that the maps $A_h \to A_k$, for $h>k>0$ are nuclear ring homomorphisms which we can assume to be inclusions.

 We choose the norms $\|\ \|_h$ in such a way that
\begin{equation}
\|\sum_{j=0}^\infty a_jt^j\|_{h}\ = \ \sum_{j=0}^\infty |a_j| \| t^j\|_h\,.
\label{1fNEW}
\end{equation}
(We remark that this  assumption is not as restrictive as it might appear: since we will anyway 
 assume that $A$ is a dual nuclear Fr\'echet space, see below, assuming in 
addition  only that  $\{ t^i | i=0,1,\ldots\}$ is a Schauder basis for $A$ would imply that the norms
\eqref{1fNEW} give $A$ its own topology, see~\cite[Theorems 10.1.2 and 10.1.4.]{P}. Furthermore, the 
Schauder basis property follows by just  assuming that the (unique) representation~\eqref{representation} converges in the topology of $A$ for every $f$.)

For the norms \eqref{1fNEW} to give us a Banach algebra it is necessary and sufficient that 
\beqn
\Vert t^{j+l}  
\Vert_h  \leq \Vert t^j \Vert_h \Vert t^l \Vert_h \qquad \text{for all} \ j,l
 \label{1f}
\eeqn
as is easy to show. The  Banach algebra $A_h$  then consists of $\sum_{j=0}^\infty a_jt^j$ such that $\|\sum_{j=0}^\infty a_jt^j\|_{h} = \sum_{j=0}^\infty |a_j| \| t^j\|_h$ is finite.

The simplest such an $A$ is given by germs of holomorphic functions at the origin. In that case
\begin{equation*}
\|\sum_{j=0}^\infty a_jt^j\|_{h}\ = \ \sum_{j=0}^\infty |a_j| h^j \qquad \text{and then} \  \ \| t^j\|_h=h^j
\,.
\end{equation*}

For any $h$ we can also form the following two topological rings:
\begin{equation*}
A_{\wcheck  h}\ = \ \varinjlim_{k>h} A_k
\end{equation*}
and
\begin{equation*}
A_{\widehat h} = \ \varprojlim_{h<k} A_k\,.
\end{equation*}
As we have assumed that the maps $A_h \to A_k$, for $h>k>0$ are nuclear ring homomorphisms we see that $A_{\wcheck  h}$ is a DNF (dual nuclear Fr\'echet ) ring and the $A_{\widehat h}$ is a NF (nuclear Fr\'echet ) ring. In our example of germs of holomorphic functions the ring $A_{\wcheck  h}$ is the ring of holomorphic functions on the (compact) closed disk or radius $h$ and the ring $A_{\widehat h}$ is the ring of holomorphic functions on the open disk or radius $h$. The $A_{\wcheck  h}$  and $A_{\widehat h}$ are  the analogues of these familiar constructions and  the ring $A_{\wcheck  h}$ will play an important role in the rest of the paper. 

\subsection{The assumptions on the families $A_h$}

We will now formulate precise conditions on the $\Vert t^j \Vert_h$ which will be in force for the rest of the paper. We will first list all the conditions and then explain their meaning. To that end let us write 
\beqn
R(h,j) := \frac{\Vert t^{j+1} \Vert_h}{ \Vert t^{j} \Vert_h} .
\eeqn
and
\beqn
\label{Nj}
N_j : [0,S] \to \bbR^+ \ , \ \  N_j(h) := \Vert t^j \Vert_h^2 
\eeqn
for every $j \in \bbN$.

\begin{subequations}
\label{cond}
First of all, in order for the $A_h$ to be Banach algebras we require
\bea
\Vert t^{j+l}  
\Vert_h  \leq \Vert t^j \Vert_h \Vert t^l \Vert_h \qquad \text{for all} \ j,l .\qquad  \text{(Banach algebra)}
\eea
For the purposes of the arguments we furthermore normalize things so that
\bea
 \Vert t^j \Vert_h \leq 1 \ \  \text{and}\ \ R(h,j) \leq 1 \ \ \text{for all} \ \ j\in\bN \,.\qquad  \text{(normalization)}  \label{22NEW}
\eea
In particular, the sequence $ \Vert t^j \Vert_h$ is decreasing. 

We also assume that
\bea
{
R(h,j) \to 0  \ \ \mbox{as} \ j \to \infty. \qquad \text{(locality)}
} \label{2c}
\eea
This condition implies that the  $A_h$ are local rings: the numbers $\Vert t^j\Vert_h$ have to decay
fast enough, faster than exponentially, as $j \to \infty$.  We could relax this condition a bit, but it simplifies the discussion to pose it. Basically it only excludes the classical case of germs of holomorphic functions. 

The first crucial assumption is the following:  for every pair $h < k$ there exists a constant $
K_{h,k} >0$  such that for every $j$,
\bea
{
 \Vert t^j \Vert_h \leq K_{h,k} \min\{ j^{-1} , R(k,j)\}
 \Vert t^j \Vert_k. \qquad \text{(controlled nuclearity)}
} \label{1g}
\eea
In particular, we are assuming $\Vert t^j \Vert_h \to 0 $ as $h \to 0$, for
every $j$.

Finally, as a second crucial assumption we assume that every $N_j$ is two times continuously differentiable, decreasing, and that
\bea
{
- \frac{ d^2  }{dh^2} \log N_j(h) \geq 
\frac{1}{h}\frac{ d  }{dh} \log N_j(h) } \qquad \text{(subharmonicity)}
\label{3j}
\eea
for all (small enough) $h \in (0, S]$ and for all $j$. 
\end{subequations}

The first three conditions are rather straightforward. We will discuss briefly the meaning of the last two conditions. 

Let us unravel the meaning of condition~\eqref{1g}. It is a ``controlled" nuclearity 
requirement. More precisely, considering the inequality with the $j^{-1}$ term guarantees 
 nuclearity  of the maps $A_h \to A_k$.  Moreover it puts a bound on the nuclearity which we will make crucial use of in section~\ref{L2 arguments} where we have to pass between our norms and $L^2$-norms. Considering the inequality with the 
 $R(C,j)$ term gives us the following inequality:
 \bea
 \label{boundedness}
 \Vert t^j \Vert_h \leq K_{h,k}  
 \Vert t^{j+1} \Vert_k. 
 \eea
This condition is needed in order to prove lemma~\ref{DVR}, i.e., that $A_{\wcheck  h}$ is a DVR. It expresses some form of nuclearity in families. Therefore we call condition~\eqref{1g} ``controlled" nuclearity .

Finally, condition~\eqref{3j} guarantees that we have enough analogues of pseudoconvex domains. In the classical setting Stein submanifolds have a cofinal family of neighborhoods which are Stein. In our setting we need an analogue of this statement. To have such cofinal families we have to impose condition~\eqref{3j}. In a technical sense it gives us enough plurisubharmonic functions to carry out the $L^2$-analysis of H\"ormander in our setting. This is done is section~\ref{L2 arguments}.

{\bf Examples of functions satisfying all conditions.}
\eqn
& & (i) \ \frac{1}{j!} h^{j^\gamma} \, , \ \gamma \geq 1 , \nonumber \\ 
& & (ii) \  j^{-k} h^{j^\gamma}  \, , \ k = 1,2,\ldots , \ 
\gamma > 1, \ h \ \mbox{small enough}, 
\nonumber \\ 
& &(iii)  \  e^{-j^k} h^{j^\gamma}  \, , \ k = 1,2,\ldots , \ \gamma > k,
\nonumber \\ 
& & (iv) \  \frac{1}{j!} e^{ - \gamma j /h}  \, , \ \gamma \geq 1 ,
\nonumber \\ 
& & (v) \  \frac{1}{j!} e^{ (1- \gamma^j )/h}  \, , \ \gamma \geq 2 . \ \ \ 
\eneqn

In the rest of the section we will prove basic facts about the sheaves $\cA_X$. In this section we do not make use of the subharmonicity condition~\eqref{3j}.
We start with:
\begin{lem}
\label{DVR}
The ring $A_{\wcheck  h}$ is a DVR 
\end{lem}
\begin{proof}
First, the condition~\eqref{2c} implies that the
 rings $A_{h}$ are local ring with maximal ideal $\fm_h=\{\sum_{j=0}^\infty a_jt^j\in A_h\mid a_0=0\}$. Hence, $A_{\widecheck h}$ is a local ring with maximal ideal $\widecheck\fm_h=\{\sum_{j=0}^\infty a_jt^j\in A_{\widecheck h}\mid a_0=0\}$. It suffices to show that $\fm_h=(t)$. Let $f\in \wcheck\fm_h$. Then there is a $k>h$ such that $f\in A_k$ and of course $f(0)=0$. Let us write $f(t) = tg(t)$ with $g(t)=\sum_{j=0}^\infty a_jt^{j-1}\in \widehat A$. We will show that  $g(t)\in A_{l}$ for any $l$ such that $k>l>h$. As $f\in A_k$ we see that
\beq\label{series 1}
\|f\|_k=\|\sum_{j=0}^\infty a_jt^j\|_{_k}\ = \ \sum_{j=0}^\infty |a_j| \|t^j\|_k <\infty\ \ \,.
\eeq
Let us now choose any $l$ such that $h<l<k$. To show that  $\|g(t)\|_{l}<\infty$ we consider:
\begin{equation}
\label{series 2}
\begin{gathered}
\|g(t)\|_{l}= \sum_{j=0}^\infty |a_j| \|t^{j-1}\|_{l}
\end{gathered}
\end{equation}
Comparing the series~\eqref{series 2} to  the series~\eqref{series 1} and using the condition~\eqref{1g} in the form of~\eqref{boundedness} 
we conclude that  $g(t)\in A_{l}$.
\end{proof}

As the spectrum of $A_h$ then consists of the origin only we see by the spectral radius formula that 
\beq
\label{spec radius}
\| t^n \|_h \ = \ (\epsilon_n)^n \qquad \text{where} \qquad \lim_{n\to\infty} \epsilon_n = 0 \,.
\eeq
\begin{rmk}
In the case of our motivating example~\eqref{C norm} we see that the $\epsilon_n$ is essentially proportional to $\frac 1 n$ by the Stirling formula. 
\end{rmk}
Let us recall that we have assumed that $\|t\|_h\leq 1$. Then $\| t^n \|_h \leq \|t^{n-1} \|_h \| t\|_h\leq  \|t^{n-1} \|_h\leq 1$. Now,
\beqn
\epsilon_n = (\| t^n \|_h)^{\frac 1 n}\leq (\| t^{n-1} \|_h)^{\frac 1 n}\ = \  (\| t^{n-1} \|_h)^{\frac 1 {n-1}} (\| t^{n-1} \|_h)^{\frac  {n-1} n}\leq \epsilon_{n-1}\,.
\eeqn
Thus, we conclude that 
\beq
\label{decreasing}
\text{The sequence $\epsilon_n$ is decreasing}\,.
\eeq

Let us now come back to analyze the rings $\cA_X$. We write $\cA_X^h = A_h\hat\otimes \cO_X$ for the sheaf of holomorphic functions with values in $A_h$ and we write $\cA_X^{\wcheck h}= A_{\wcheck  h}\hat\otimes \cO_X$ for the sheaf of holomorphic functions with values in $A_{\wcheck  h}$.

\begin{lem}
The stalks $\cA_{X,x}$ and  $\cA_{X,x}^{\wcheck h}$ are local rings. 
\end{lem}

\begin{proof}
The argument is the same in both cases, so we work with $\cA_{X,x}$. The maximal ideal $\fm_{A,x}\subset \cA_{X,x}$ consists of functions $f\in \cA_{X,x}$ such that $f(x)\in \fm$. To prove that $\cA_{X,x}$ is local we have to show that any $f\notin \fm_{A,x}$ is invertible.   Any element $f\in \cA_{X,x}$ is represented by a series
\begin{equation*}
f\ = \ \sum a_\alpha x^\alpha \qquad a_\alpha \in A\,;
\end{equation*}
here we have replaced $X$ by $\bC^n$ and assumed that $x$ is the origin. This series converges in some neighborhood $U$ of the origin. Let us now restrict $f$ to a smaller neighborhood $V$ such that $\bar V \subset U$. As the ring $A$ is equipped with a direct limit topology, there is an $h$ such that $f|_{\bar V}\in \cA_h(\bar V)$. Note also that if $f\notin \fm_{A,x}$ then the first term $a_0\in A^*=A-\fm$, the units in $A$. Thus, 
\begin{equation*}
a_0^{-1}f\ = \ 1+ \sum a_\alpha x^\alpha\,;
\end{equation*}
Now, as $\sum a_\alpha x^\alpha$ vanishes at the origin, we can, by making the neighborhood $U$ smaller if necessary, assume that $\|\sum a_\alpha x^\alpha\|_h <1$. Thus $a_0^{-1}f$ and hence $f$ is invertible. 

\end{proof}

\begin{rmk}
The sheaves $\cA_X$ and $\cA_X^{\wcheck h}$ are defined in the same way. For emphasis we will make statements in both cases, but of course for the proof we can just think in terms of $\cA_X$. 
\end{rmk}

\begin{thm}
\label{regular local}
The stalks $\cA_{X,x}$ and  $\cA_{X,x}^{\wcheck h}$ are regular local rings. 
\end{thm}

To prove the theorem we work locally so that we can assume that $X=\bC^n$ and we choose local coordinates $x_1, \dots ,x_n$ on $\bC^n$ such that the point $x$ corresponds to the origin. We write $f$ in local coordinates
\begin{equation*}
f \ = \  \sum a_i t^i \qquad a_i\in \cO_{X,x}\,.
\end{equation*}
We prove this statement in the standard manner by first proving an appropriate Weierstrass division theorem. We follow the classical argument as presented in \cite[Chapter 2]{GrRe2}. We will also try to stick to the notation there as closely as possible.  As a first step we argue that by a change of coordinates we can write $f$ so that it is $t$-regular, i.e., that there is a $b$ such that 
\begin{equation*}
a_0(0)= \dots = a_{b-1}(0)=0 \ \  \ \text{and} \ \ \ a_b(0)\neq 0\,.
\end{equation*}
As we work locally, the function $f$ is holomorphic on some closed  polydisk of radius $\rho=(\rho_1,\dots, \rho_n)$ and on that polydisk  $f\in\cA_{X}^h$ for some $h$. Thus we can consider the following norm $\| \ \|_\rho$ on $f$ by 
\begin{equation}
\label{norm}
\| f\|_\rho\ = \ \sum_{i,\alpha}|a_{\alpha,i}| \rho^\alpha \|t^i\|_h \ \ \text{where}\ \ \ a_i = \sum a_{\alpha,i} x_1^{\alpha_1}\dots x_n^{\alpha_n}\ \ \  \rho^\alpha = \rho_1^{\alpha_1}\dots \rho_n^{\alpha_n}\,.
\end{equation}
We now substitute 
\beq
\label{coordinate change}
w_1 = x_1 + c_1 t, \dots , w_n = x_n + c_n t
\eeq
to get the new coordinates $w_1, \dots, w_n, t$. A generic choice of small such $c_i$ will make $f$ regular in $t$. It is perhaps good to note that outside of special cases we can only make $f$ regular in $t$ and not in any of the other variables. 

We now write $f$ in this new set of variables as 
\begin{equation*}
 f = \ \sum_{i,\alpha} b_{\alpha,i}w_1^{\alpha_1}\dots w_n^{\alpha_n} t^i=\sum_{i} b_{i} t^i\,.
\end{equation*}
In these new coordinates we of course will get a different $\rho$ for the radius of convergence. However, it is important that we do not change $h$. 

We perform this substitution one variable at a time and keep the other variables fixed. So, we are reduced in the one variable situation in the base and we write $x=x_1$, $w=w_1= x +ct$ and the $\rho=\rho_1$. Our $f$ can now be written as
\beq
\label{before substitution}
f = \sum_{i,j} a_{i,j} x^jt^i \,.
\eeq
and after the substitution
\beq
\label{after substitution}
\begin{gathered}
 f = \sum_{i,j} a_{i,j} (w+ct)^j t^i = \sum_{i,j} a_{i,j} \sum_{p+q=j} {\binom {p+q} p} w^p c^q t^{i+q} = 
 \\
 =\sum_{i,p,q} a_{i,p+q}  {\binom {p+q} p} w^p c^q t^{i+q}\,.
 \end{gathered}
\eeq
We now consider~\eqref{before substitution} and we obtain
\beqn
\| f\|_\rho = \sum_{i,j} |a_{i,j}| \rho^j \epsilon_i^i<\infty  \,.
\eeqn
We then write
\beqn
 \la_{ij}=|a_{i,j}| \rho^j \epsilon_i^i \ \ \text{and then}\ \ |a_{i,j}| =\la_{ij} \rho^{-j}\epsilon_i^{-i}\ \ \text{and} \ \  \sum_{i,j} \la_{ij}<\infty \ \,.
\eeqn
We furthermore set
\beq
\label{A1}
\la_i = \max\{\la_{ij}\mid j=0, \dots\} \ \ \text{and we still have} \ \  \sum\la_i < \infty\,.
\eeq
This allows us to estimate the norm in~\eqref{after substitution} after the substitution
\beqn
\begin{multlined}
\| f \|_r
 =\sum_{i,p,q} |a_{i,p+q}|  {\binom {p+q} {p}} r^p c^q \|t\|_{h}^{i+q}= 
 \\\sum_{i,p,q}\la_{i,p+q}  {\binom {p+q} {p}}\rho^{-p-q} \epsilon_i^{-i}r^p c^q {\epsilon_{i+q}}^{i+q}\leq
 \\
\sum_{i,q}\la_i  \left({\frac {\epsilon_{i+q}}{\epsilon_i}}\right)^{\!i}\left({\frac {c\epsilon_{i+q}}{\rho}}\right)^{\!q}\sum_p{\binom {p+q} {p}} \left({\frac {r}{\rho}}\right)^{\!p} =
 \\
  =  \sum_{i,q} \la_i   \left({\frac {\epsilon_{i+q}}{\epsilon_i}}\right)^{\!i}\left({\frac {c\epsilon_{i+q}}{\rho}}\right)^{\!q}\left(1-{\frac {r}{\rho}}\right)^{\!-q-1} 
\\
=  \left(1-{\frac {r}{\rho}}\right)^{\!-1} \sum_{i,q} \la_i \left({\frac {\epsilon_{i+q}}{\epsilon_i}}\right)^{\!i}\left({\frac {c\epsilon_{i+q}}{\rho-r}}\right)^{\!q} <\infty
\end{multlined}
 \eeqn
because by~\eqref{decreasing} we have $\epsilon_{i+q}\leq \epsilon_{i}$, by~\eqref{spec radius} we have $\epsilon_{i+q}\to 0$ when $i+q\to \infty$, and by~\eqref{A1} we have $\sum \la_i<\infty$. 

We proceed in this manner one variable at a time keeping the others constant. We thus have reached the following conclusion:
\beq
\label{t-regular}
\begin{gathered}
\text{Given $f\in\cA_{X,x}^h$ we can make it $t$-regular by a coordinate change~\eqref{coordinate change}}\,.
\end{gathered}
\eeq
Let us consider $f\in\cA_{X,x}$. The function $f$ is holomorphic on some closed  polydisk of radius $\rho=(\rho_1,\dots, \rho_n)$ and on that polydisk  $f\in\cA_{X}^h$ for some $h$. We consider the norm $\|\ \|_\rho$ defined in~\eqref{norm}. As before, we write
\begin{equation*}
f \ = \  \sum a_i t^i \qquad a_i\in \cO_{X,x}\,.
\end{equation*}
and then we write 
\begin{equation*}
\hat f \ = \  \sum_0^{b-1} a_i t^i \ \ \  \tilde f \ = \  \sum_b^{\infty} a_i t^{i-b}\ \ \ f=\hat f + \tilde f t^b\,.
\end{equation*}
Note that although, of course, $\tilde f t^b\in \cA_{X,x}^h$, the function $\tilde f$ does not necessarily lie in $\cA_{X,x}^h$. However, it does lie in any $\cA_{X,x}^k$ for $k<h$. We  see this by applying the following general principle inductively:
\begin{lem}
Let $g\in \cA_X^{\wcheck h'}(U)$ and assume that $g(x)\in \fm_{h'}$ for all $x\in U$. Then have $g=t \tilde g$ with  $\tilde g\in \cA_X^{\wcheck h'}(U)$. 
\end{lem}
\begin{proof}
Let us then consider the map $t:A_{\wcheck h'} \to A_{\wcheck h'}$ given by multiplication by $t$. By our hypotheses the map $t: A_{\wcheck h'} \to\wcheck\fm_{h'}$ is continuous and a bijection. As both $A_{\wcheck h'}$ and $\wcheck\fm_{h'}$ are DNF, the open mapping theorem applies and hence $t: A_{\wcheck h'} \to \fm_{h'}$ is an isomorphism. Thus, $t^{-1}: \wcheck\fm_{h'} \to A_{\wcheck h'}$ is a well defined continuous map and so we can write $g=t \tilde g$  where $\tilde g\in \cA_X^{\wcheck h'}(U)$.
\end{proof}
We apply this lemma after first passing to $\cA_{X,x}^{\wcheck h'}$ with $k<h'<h$ and repeat the process to extract the $b$ copies of $t$ from $\tilde f t^b$. 
We now also have
\beq
\label{estimate}
\|\hat f\|_\rho \leq \| f\|_\rho \qquad \| \tilde f\|_\rho \leq \| f\|_\rho \|t^b\|^{-1}\,.
\eeq
Recall that any statements like this we make about norms are valid as long as both sides are defined. In particular, the latter inequality holds for any $t$-norm $k$ where $k<h$. 

We are now ready to state and prove the Weierstrass division theorem:
\begin{prop}
If $g\in \cA_{X,x}$ has order $b$ in $t$ then for any $f\in \cA_{X,x}$ there exists a $q\in \cA_{X,x}$ and an $r\in\cO_{X,x}[t]$ with degree of $r$ in $t$ less than $b$ such that $f=qg+r$. 
\end{prop}
\begin{proof}
To prove this result, we proceed just as in the classical case. The argument in \cite[Chapter 2, \S1]{GrRe2} can be adopted to our situation and we briefly indicate the necessary changes trying to stick as close to the notation there as possible. First of all, we can view the stalk $\cA_{X,x}$ as a direct limit of Banach algebras $B^h_\rho$, where $\rho=(r_1, \dots, r_n)$ is a sequence of positive real numbers and the norms are given by
\begin{equation*}
\text{For} \  \ f=\sum b_\al x^\al  \ \ \text{we set}  \  \ \|f\|_{h,\rho} \ = \ \sum \|b_\al\|_h\, r_1^{\al_1}\dots r_n^{\al_n}
\end{equation*}
just as in \cite[Chapter 2, \S1]{GrRe2} for the case $A_k=\bC$ and as we have done above. Note that because $g$ is $b$-regular in $t$ then  $\tilde g$ is invertible. 
We first note that we can choose $\rho$ sufficiently small so that $g,\hat g, \tilde g$, and $\tilde g^{-1}$ lie in $B^h_\rho$ for some $h$. We will further adjust $h$ and $\rho$ so that we also have $f\in B^h_\rho$.
Note that
\beqn
\|t^b-g\tilde g^{-1}\| = \|\hat g \tilde g^{-1}\| \leq \|\hat g\|\| \tilde g^{-1}\| 
\eeqn
Now, we have for $g = \sum a_i t^i$ that $a_0(0)=\dots =a_{b-1}(0)=0$ and hence  $\|\hat g\|$ can be made arbitrarily small by shrinking $\rho$. As $\tilde g^{-1}$ is invertible, its norm can be bounded away from zero. Thus, for any $\epsilon>0$ we can, by shrinking $\rho$, arrange things so that 
\beqn
\|t^b-g\tilde g^{-1}\| <\epsilon \|t^b\| \,.
\eeqn
We now set
\beqn
v_0 = f, \dots , v_{j+1} = (t^b-g\tilde g^{-1})\tilde v_j = -\hat g \tilde g^{-1}\tilde v_j, \dots 
\eeqn
By~\eqref{estimate} we get that 
\beqn
\|\tilde v_j\|_\rho \leq \| v_j\|_\rho \|t^b\|^{-1}\qquad \text{and then}\qquad \|v_{j+1}\|\leq\epsilon \|v_j\|_\rho\,.
\eeqn
Recall that for a given $h$ the estimate~\eqref{estimate} holds for any $k<h$ but not necessarily for $h$ itself. We make use of the estimate every time we form a $\tilde v_j$ but for any $k$ we can choose $h>k_0>k_1>\dots > k_j >\dots >k$. When we form $\tilde v_j$ we pass from $k_{j-1}$ to $k_j$. Thus, the estimates above holds for any $k<h$. 
We now set, for $k<k'<h$
\beqn
v = \sum_0^\infty v_j\in B_\rho^{k'}\ \ q=\tilde g^{-1}\tilde v\in B_\rho^k\ \ r=\hat v\in B_\rho^k\,.
\eeqn
Then
\beqn
f = \sum _0^\infty(v_j-v_{j+1}) = \sum_0^\infty(g\tilde g^{-1}\tilde v_j + \tilde v_j) = qg+r\,.
\eeqn
\end{proof}

To prove that $\cA_{X,x}$ is Noetherian is clearly suffices to show that:
\begin{equation*}
\text{For any $f\in \cA_{X,x}$ the ring $\cA_{X,x}/\cA_{X,x}f$ is Noetherian}\,.
\end{equation*}

We perform a change of coordinates as in~\eqref{coordinate change} so that by~\eqref{t-regular} we can assume that $f$ is $t$-regular.

The Weierstrass preparation theorem follows from the Weierstrass division theorem formally and it implies, with our hypotheses on $f$, that we have
\begin{equation*}
f=ug \ \ \text{$u\in\cA_{X,x}$ is a unit}\ \  g\in\cA_{X',x}[x_n]\,.
\end{equation*}
Therefore
\begin{equation*}
\cA_{X,x}/\cA_{X,x}f\  \cong\ \cA_{X,x}/\cA_{X,x}g\  \cong \ \cA_{X',x}[x_n]/\cA_{X',x}[x_n]g\,.
\end{equation*}
But, by induction, we conclude that $\cA_{X',x}[x_n]$ is Noetherian and hence so is $\cA_{X,x}/\cA_{X,x}f$.
To see the second isomorphism we apply the division theorem again.

\begin{prop}
\label{faithful flatness}
Let $k<h$ then $\cA_{X,x}^{\hat k}$ is faithfully flat over $\cA_{X,x}^{\hat h}$\,.
\end{prop}
\begin{proof}
We make use of the appendix of \cite{Se} where the notion of (faithful) flatness was originally introduced. The local ring is  $\cA_{X,x}^{\hat h}$ is a subring of  $\cA_{X,x}^{\hat k}$ and their formal completions coincide. Now we use the fact that the formal completion $\hat B$ of a regular local ring $B$ is faithfully flat over $B$. This implies that $\cA_{X,x}^{\hat k}$ is faithfully flat over $\cA_{X,x}^{\hat h}$. 
\end{proof}

\begin{prop}
\label{coherence}
The sheaf $\cA_X$ is coherent. 
\end{prop}

\begin{proof}
Note that $\cA_X$ is naturally a filtered ring and the associated graded $\operatorname{gr} \cA_X \cong \cO_X[t]$ is a coherent. Now, we just checked that $\cA_{X,x}$ is Noetherian and it is Zariskian, because it is a local ring. Then $\cA_X$  is coherent by \cite[Chapter II, proposition 1.4.1]{Sch}.
\end{proof}

\section{The case of varying levels}
\label{varying levels}

In our proof of the Cartan theorems we have to allow the level $h$ to vary. We introduce this generality already here
as the proofs are the same as in the constant case. 

Given a continuous function $h: X \to \bR_+$, we can consider the corresponding sheaf $\cA^{h}_X$ on $X$ defined as follows:
\begin{equation}
    \cA^{h}_X(U) = \{f: U \to \hat A=\bC[[t]] \mid f \ \text{holomorphic and $f(x)\in A_{h(x)}$ for all $x\in U$}\}\,.
\label{4.1}
\end{equation}
We also have the corresponding sheaf $\cA^{\wcheck h}_X$ obtained as a direct limit
\begin{equation}
\label{4.2}
\cA^{\wcheck h}_X \ = \ \varinjlim_{h_1 > h} \cA^{h_1}_X\,;
\end{equation}
here $h_1 > h$ means that $h_1(x) > h(x)$ for all $x$.

Let us consider the stalks of the sheaves $\cA^{\wcheck h}_X$. We claim that 
\begin{equation}
\label{stalks}
\cA^{\wcheck {h}}_{X,x} = \varinjlim_{U\ni x} \cA^{\wcheck h}_X(U) = \cA^{
\wcheck {h(x)}}_{X,x}\,.  
\end{equation}
To see this, note that 
\begin{equation*}
\begin{gathered}
\text{For any $h_1>h$ there exists a neighborhood $U_1$ of $x$}
\\
\text{such that $h_1(y)>h(x)$ for any $y\in U_1$}\,.
\end{gathered}
\end{equation*}
Now, any  $s\in\varinjlim_{U\ni x} \cA^{\wcheck h}_X(U)$ is given by  $s\in\cA^{\wcheck h}_X(U)$ for some $U\ni x$. Restricting $s$ further to a neighborhood $V\ni x$ with compact closure we conclude that there is an $h_1>h$ such that $s\in\cA^{h_1}_X(V)$. If necessary we shrink $V$ further so that it is contained in $U_1$ given by~\eqref{stalks}. Thus, $s$ gives rise to an element in $\cA^{h_1(x)}_X(V)$ and hence an element in $\cA^{
\wcheck {h(x)} 
}_{X,x}$. This gives a map $\cA^{\wcheck {h}}_{X,x}\to \cA^{
\wcheck {h(x)} 
}_{X,x}$. It is now easy to see that this map is an isomorphism.

\begin{rmk}
Note that we do not have an analogous statement for $\cA^{h}_X$. 
\end{rmk}

We have
\begin{prop}
\label{varying coherence}
The sheaf $\cA^{\wcheck h}_X$ is coherent. 
\end{prop}
The proof of this proposition is the same as the proof of~\ref{coherence}.

Let us now consider an $\cA^{\wcheck h}_X$-coherent sheaf $\cF$. First, given any $x\in X$ and using the fact that the local ring $\cA_{X,x}$ is regular we obtain a resolution $\cL_\cdot^x$ of $\cF_x$  of length $n+1$ with the $\cL_k^x$ free $\cA_{X,x}$-modules of finite rank. Thus, there exists a neighborhood $V_x$ of $x$ such that the connecting homomorphisms $\del_i:\cL_i^x\to \cL_{i-1}^x$ are defined  on $V_x$ and so we can view $\cL_\cdot^x$ as a complex of free $\cA_{V_x}$-modules of finite rank on $V_x$. The homology groups $H_k(\cL_\cdot^x)$ are coherent sheaves on $V_x$ such that the stalks $H_k(\cL_\cdot^x)_x=0$. Thus, possibly by shrinking $V_x$ the complex $\cL_\cdot^x$ is a resolution of $\cF|_{V_x}$.

 In particular, given any $x\in X$ there is neighborhood $V_x$ of $x$ such that  $\cF|_{V_x}$ has a free resolution $\cL_\cdot$. We can shrink the $V_x$ so that the new $V_x$ has compact closure in the old one. This allows us to assume that the maps in the complex $\cL_\cdot$ are defined over $\cA^{h'}_X$ for an $h'>h$. So, if $\cL_\cdot$ is given by 
\beqn
\cL_\cdot = \dots \to (\cA_{V_x}^{\wcheck h})^{\oplus p_k} \to \dots\to (\cA_{V_x}^{\wcheck h})^{\oplus p_1} \to (\cA_{V_x}^{\wcheck h})^{\oplus p_0} 
\eeqn
then for any $k$ with $h' \geq k> h$ we can form the complex $\cL^{k}_\cdot$
\beqn
\cL_\cdot^{k} = \dots \to (\cA_{V_x}^{k})^{\oplus p_k} \to \dots\to  (\cA_{V_x}^{k})^{\oplus p_1} \to (\cA_{V_x}^{k})^{\oplus p_0};
\eeqn
note that $\cL_\cdot^{k}$ it is just a complex and we make no claims about vanishing of any of its cohomology groups. To make it a resolution we pass to the direct limit as follows. For any $h'>h_1>h$ we then have 
\beqn
\cL_\cdot^{\wcheck h_1} = \varinjlim_{ h' \geq k> h_1} \cL_\cdot^{k} = \dots \to (\cA_{V_x}^{\wcheck h_1})^{\oplus p_k} \to \dots\to  (\cA_{V_x}^{\wcheck h_1})^{\oplus p_1} \to (\cA_{V_x}^{\wcheck h_1})^{\oplus p_0} 
\eeqn

Now, clearly, 
\beqn
\cL_\cdot = \cA_{V_x}^{\wcheck h}\otimes_{ \cA_{V_x}^{\wcheck h_1}} \cL_\cdot^{\wcheck h_1} \,.
\eeqn
Passing to the level of stalks we obtain
\beqn
(\cL_\cdot)_x = \cA_{X,x}^{\wcheck h}\otimes_{ \cA_{X,x}^{\wcheck h_1}} (\cL_\cdot^{\wcheck h_1})_x \,.
\eeqn
As $\cA_{X,x}^{\wcheck h}$ is faithfully flat over $ \cA_{X,x}^{\wcheck h_1}$ we conclude that $H_k((\cL_\cdot^{\wcheck h_1})_x)=0$ for $k\geq 1$ because that is the case for $(\cL_\cdot)_x $ and then, finally, that
\beqn
H_k(\cL_\cdot^{\wcheck h_1})=0 \qquad \text{ for $k\geq 1$}\,.
\eeqn
We now set
\beqn
\cF_{V_x}^{\wcheck h_1} \ = \ H_0(\cL_\cdot^{\wcheck h_1})\,.
\eeqn
Note that, by construction, we also then have for $h_1\geq h_2 \geq h$
\beqn
\cF_{V_x}^{\wcheck h_2 }\cong\cA_{V_x}^{\wcheck h_2}\otimes_{\cA_{V_x}^{\wcheck h_1}} \cF_{V_x}^{\wcheck h_1}\,.
\eeqn
In particular, 
\beqn
\cF|_{V_x}\cong\cA_{V_x}\otimes_{\cA_{V_x}^{\wcheck h_1}} \cF_{V_x}^{\wcheck h_1}\,.
\eeqn
By faithful flatness we have a canonical inclusion
\beqn
 \cF_{V_x}^{\wcheck h_1} \subset \cF|_{V_x} \,.
\eeqn

The constructions that we have just performed depend on the choice of the initial resolution $\cL_\cdot$. We will next study this dependence. We can do this in two ways. One is to work to with compact neighborhoods of $x$ to begin with or we work with the open neighborhoods $V_x$, but always think of them as being equipped with an open neighborhood $U_x$ of $x$ such that $U_x \subset \bar U_x \subset V_x$ with $\bar U_x$ compact. Let us choose a cover of $X$ by the $U_x$, with $x\in I$ such that the only finitely many $V_x $ have a non-empty intersection. 

Now, for each $V_x$ we obtain a particular function $l_x$ as above. Furthermore, for each pair of $x,y$ such that $U_x\cap U_y$ is non-empty, we can compare the resolutions $\cL_\cdot(x)$ and $\cL_\cdot(y)$ by first restricting them to $V_x\cap V_y$. So, we have on  $V_x\cap V_y$
\beqn
\begin{CD}
 \dots @>>>(\cA_{V_x\cap V_y}^{\wcheck h})^{\oplus p_1} @>>> (\cA_{V_x\cap V_y}^{\wcheck h})^{\oplus p_0} @>>> \cF|_{V_x\cap V_y} @>>> 0
 \\
 @.    @VV{\al_1}V @VV{\al_0}V @| @.
 \\
  \dots @>>>(\cA_{V_x\cap V_y}^{\wcheck h})^{\oplus q1_1} @>>> (\cA_{V_x\cap V_y}^{\wcheck h})^{\oplus q_0} @>>> \cF|_{V_x\cap V_y} @>>> 0
\end{CD}
\eeqn
where the vertical arrows $\al_i$ are lifts of the identity map. Now the maps $\al_i$ and the maps to the opposite direction are defined on some particular  level $l_{x,y}$ with $l_{x,y}>h$. We now find a $k_x$ on $U_x$ such that 
\beqn
\text{ $l_x \geq k_x > h$ and $l_{x,y}\geq k_x$, $l_y\geq k_x$ for all $y$ such that $U_x\cap U_y$ not empty}\,.
\eeqn
For any $k$ such that $k_x>k>h$ we can now construct the sheaf $\cF_{U_x}^{\wcheck k}$ by making use of the resolution $\cL_\cdot(x)$. Furthermore, if $U_x\cap U_y$ in not empty then, by our choice of $D_x$  we get an isomorphism 
between $\cF_{U_x}^{\wcheck k}|_{U_x\cap U_y}$  and $\cF_{U_y}^{\wcheck k}|_{U_x\cap U_y}$. This isomorphism becomes, by construction,  the identity when we extend scalars to $\cA^{\wcheck h}_X$. This allows us to identify $\cF_{U_x}^{\wcheck k}|_{U_x\cap U_y}$ and $\cF_{U_y}^{\wcheck k}|_{U_x\cap U_y}$ as sub sheaves of $\cF|_{U_x\cap U_y}$. So we can glue the sheaves to obtain a sheaf $\cF_{U_x\cup U_y}^{\wcheck k}$. 

Finally, we now choose a function $k$ such that $k>h$ and $k_x\geq k$ for all $x\in X$. Once we have done this then for any $h_1$ with $k>h_1>h$ we have constructed sheaves $\cF^{\wcheck h_1}$ which coincide with the $\cF_{U_x}^{\wcheck h_1}$ on $U_x$ and which thus satisfy
\beqn
\cF\cong\cA_X^{\wcheck h}\otimes_{\cA_X^{\wcheck h_1}} \cF^{\wcheck h_1 }\,.
\eeqn
Furthermore, we have a canonical inclusion:
\beq
\label{inclusion}
\cF^{\wcheck h_1 }\subset  \cF\,.
\eeq

\section{Cartan theorems for coherent  $\cA_X^{\wcheck h}$-modules on compact blocks}
\label{compact section}

In this section we prove Cartan's theorem for compact blocks in $\bbC^N$ for certain varying levels $h(x)$. We restrict our attention to blocks mainly for simplicity of  exposition as this is enough for our arguments in the rest of the paper. We make use of the $L^2$-methods of H\"ormander as they seem best suited to this task.

We will denote the coordinates in $\bbC^N$ by $z=(z_1, \dots, z_N)$. Recall that compact blocks $Q\subset \bbC^N$ are products of rectangles $R =\{z\in \bbC \mid a\leq \Re z \leq b \ , \ c\leq \Im z \leq d\}$.
We allow for the possibility of degenerate rectangles and argue by induction on the dimension of $Q$.  
 We write $r=\sup\{|z_1|,\dots,|z_n|\}$ and use it as a norm on $Q$. 
 
 We consider functions $h: Q \to \bbR^+$ which only depend on $r$ and which are twice differentiable as functions of $r$ satisfying the following condition:
\begin{eqnarray}
h'(r)^2 \geq   h(r)   h''(r)\qquad \text{for all $r>0$} \,.
\label{h condition}
\end{eqnarray}
Let us note that it is easy to characterize such functions. We write
\bea
\label{big H}
h(r) = \exp\big( -\int_0^r H(s)ds \big)\,.
\eea
The condition \eqref{h condition} is then equivalent to 
\bea
H(r)^2 \geq -  H'(r) + H(r)^2 \ \ \ \mbox{for all } r > 0  \label{3g}\,.
\eea
This condition holds precisely as long as $H'(r)\geq 0$, i.e., as long as $H(r)$ is non-decreasing. 

\begin{thm}
\label{compact case}
If $h$ satisfies condition~\eqref{h condition} and $\cF$ is a coherent  $\cA^{\wcheck h}_{X_\nu}$-module defined on a neighborhood of $Q$ then $\oh^k(Q,\cF)=0$ for $k\geq 1$ and $\cF$ is generated by global sections in the strong sense, i.e., there is an $n$ and a surjection $(\cA^{\wcheck h}_{Q})^{\oplus n} \to \cF$. 
\end{thm}

\begin{rmk}
Note that we can view $\cA^{\wcheck h}_{Q}$ as specifying an analogue of a compact domain. According to the proposition when $h$ satisfies condition~\eqref{h condition} this domain is holomorhically convex. We can view such domains as compact neighborhoods of $Q$ in $Q\times \bbC$. By the discussion above such holomorphically convex domains are plentiful and form a basis of all compact neighborhoods. 
\end{rmk}

We prove the following theorem using Dolbeault cohomology and standard $L^2$ methods. 
\begin{prop}
If $Q$ is a compact block and $h$ satisfies condition~\eqref{h condition} then $\oh^k(Q,\cA_{Q}^{\wcheck h})=0$ for $k\geq 1$. Furthermore, if $E$ is a locally free $\cA^{\wcheck h}_{\bbC^N}$-module of finite rank (i.e., an $\cA^{\wcheck h}_{\bbC^N}$ vector bundle) defined in a neighborhood of $Q$ then $\oh^k(Q,E)=0$ for $k\geq 1$.
\label{cohomology vanishing}
\end{prop}
 First of all, we prove the theorem by proving that the (global) Dolbeault complex is exact. To explain the Dolbeault complex let us return temporarily to the case of an arbitrary complex manifold $X$. First of all let us consider the classical Dolbeault complex on $X$:
\beq
\label{classical Dolbeault}
C^\infty_X=\Omega^{0,0}_X\xrightarrow{\bar\del} \Omega_X^{0,1}  \xrightarrow{\bar\del} \dots \xrightarrow{\bar\del} \Omega^{0,q}_X \xrightarrow{\bar\del}
\eeq
which gives us a resolution of $\cO_{X}$. Let us write $\cC^{\infty,\wcheck h}_{X}$ for the smooth version of $\cA^{\wcheck h}_{X}$, i.e., we first set:
\begin{equation*}
\begin{gathered}
    \cC^{\infty,h}_X(U) = \{f: U \to \hat A=\bC[[t]] \mid f \ \text{smooth and $\frac{\partial^\alpha\! f}{\partial x^\alpha}(x)\in A_{h(x)}$}
    \\
    \text{ for all $x\in U$ and all multi indices $\alpha$}\}\,.
    \end{gathered}
\end{equation*}
and then pass to a direct limit. Note that if $h(x)=k$ is constant then  $\cC^{\infty,k}_X(U)$ agrees with the usual notion of smooth functions with values in $A_k$ and then $\cC^{\infty,k}_X(U)= A_h \hat \otimes_\bC C^\infty_X(U)$. This follows from the fact, which is easy to verify in our case directly that continuity and continuous differentiability can be checked component wise (weak continuity/differentiability).

We can now form the  $\bar\del$-complex of $\cA^{\wcheck h}_{X}$ by tensoring the above complex~\eqref{classical Dolbeault} with $\cC^{\infty,\wcheck h}_{X}$ over $C^\infty_X$. Let us write
\beqn
\cA_X^{\wcheck h,(0,q)} = \cC^{\infty,\wcheck h}_{X}\otimes_{C^\infty_X} \Omega^{0,q}_X
\eeqn
and then we obtain a complex
\beq
\label{our Dolbeault}
\cC^{\infty,\wcheck h}_{X}=\cA_X^{\wcheck h,(0,0)} \xrightarrow{\bar\del} \cA_X^{\wcheck h,(0,1)}  \xrightarrow{\bar\del} \dots \xrightarrow{\bar\del} \cA_X^{\wcheck h,(0,q)} \xrightarrow{\bar\del}\,.
\eeq
Finally if $E$ is locally free $\cA^{\wcheck h}_{X}$  sheaf we write $\cE = \cC^{\infty,\wcheck h}_{X}\otimes_{\cA^{\wcheck h}_{X}} E$ for the smooth version and we also write
\beqn
\cE_X^{\wcheck h,(0,q)} = \cE \otimes_{\cC^{\infty,\wcheck h}_{X}} \cA_X^{\wcheck h,(0,q)} 
\eeqn
and we obtain a complex 
\beq
\label{bundle Dolbeault}
\cE=\cE_X^{\wcheck h,(0,0)} \xrightarrow{\bar\del} \cE_X^{\wcheck h,(0,1)}  \xrightarrow{\bar\del} \dots \xrightarrow{\bar\del} \cE_X^{\wcheck h,(0,q)} \xrightarrow{\bar\del}\,;
\eeq
the operator $\bar\del$ is well-defined because $E$ is a holomorphic $\cA^{\wcheck h}_{X}$-bundle. 
The complexes~\eqref{our Dolbeault} and~\eqref{bundle Dolbeault} are resolutions of $\cA_X^{\wcheck h}$ and $E$, respectively. One can deduce this from the exactness of~\eqref{classical Dolbeault} as follows. As exactness is a local question, the exactness of~\eqref{bundle Dolbeault} follows from that of~\eqref{our Dolbeault}. We have, passing to stalks:
\beqn
\cA_{X,x}^{\wcheck h,(0,q)} = \cC^{\infty,\wcheck h}_{X,x}\otimes_{C^\infty_{X,x}} \Omega^{0,q}_{X,x}\,.
\eeqn
Now, arguing just as for~\eqref{stalks}
 in case of $\cA_{X,x}^{\wcheck h}$, we have
\beqn
\cC^{\infty,\wcheck h}_{X,x} \ = \ \cC^{\wcheck {h(x)}}_{X,x} 
\eeqn
and, finally, by~\cite[\S 41 (6) b)]{Ko},
\eqn
\cC^{\wcheck {h(x)}}_{X,x} = \varinjlim_{U\ni x} \cC^{\wcheck {h(x)}}_X(U)  = \varinjlim_{U\ni x} A_{\wcheck {h(x)}}\hat\otimes_\bC C^\infty_X(U) = A_{\wcheck {h(x)}}\hat\otimes_\bC \varinjlim_{U\ni x} C^\infty_X(U)  = A_{\wcheck {h(x)}}\hat\otimes_\bC C^\infty_{X,x}\,.
\eneqn
Putting things together, we conclude that 
\beq
\cA_{X,x}^{\wcheck h,(0,q)} = A_{\wcheck {h(x)}}\hat\otimes_\bC C^\infty_{X,x}\otimes_{C^\infty_{X,x}} \Omega^{0,q}_{X,x} = A_{\wcheck {h(x)}}\hat\otimes_\bC \Omega^{0,q}_{X,x}\,.
\eeq
The exactness of~\eqref{our Dolbeault}, and hence that of~\eqref{bundle Dolbeault}, follow from the exactness of~\eqref{classical Dolbeault} because $A_{\wcheck {h(x)}}\hat\otimes_\bC$ is an exact functor on DF-spaces by Lemma~\ref{exactness of tensor product}. 

Finally, as $\cA_X^{\wcheck h,(0,q)}$ and $\cE_X^{\wcheck h,(0,q)}$ are modules over $C^\infty_X$ they are soft and hence acyclic. 
Thus we have reduced the proof of Proposition~\ref{cohomology vanishing} to the following:
\begin{prop}
\label{Dolbeault exactness}
If $Q$ is a compact block in $\bC^N$ and $h$ satisfies condition~\eqref{h condition} then the complex
\eqn
\cA_{\bC^N}^{\wcheck h,(0,0)}(Q) \xrightarrow{\bar\del} \cA_{\bC^N}^{\wcheck h,(0,1)}(Q)   \xrightarrow{\bar\del} \dots \xrightarrow{\bar\del} \cA_{\bC^N}^{\wcheck h,(0,q)}(Q)  \xrightarrow{\bar\del}
\eneqn
is exact in degrees $q\geq 1$. Similarly,  if $E$ is a locally free $\cA^{\wcheck h}_{\bbC^N}$-module of finite rank (i.e., an $\cA^{\wcheck h}_{\bbC^N}$ vector bundle) defined in a neighborhood of $Q$ then the complex
\eqn
\cE_{\bC^N}^{\wcheck h,(0,0)}(Q) \xrightarrow{\bar\del} \cE_{\bC^N}^{\wcheck h,(0,1)}(Q)  \xrightarrow{\bar\del} \dots \xrightarrow{\bar\del} \cE_{\bC^N}^{\wcheck h,(0,q)}(Q) \xrightarrow{\bar\del}\eneqn
is exact in degrees $q\geq 1$.
\end{prop}
We postpone the proof of this proposition to the next section. 
However, let us record the following lemma which explains the meaning of condition~\eqref{h condition} as well as our initial requirement~ \eqref{3j}. These conditions are chosen precisely so that the following lemma holds. Without plurisubharmonicity vanishing of cohomology would not hold in general. 
\BEL
\label{lem11}
If $h$ satisfies condition~\eqref{h condition} then the function 
\bea
W_j : z \mapsto - 2 \log \Vert t^j \Vert_{h(r)}  \ , \ \ z \in Q,  \label{1n}
\eea
is plurisubharmonic  for every  $j$. 
\ENL
\begin{proof}
Recall the notation from~\eqref{Nj}.
We first remark that \eqref{3j} implies, by the positivity of $N_j$ and 
a direct calculation of the derivatives,  that
\bea
\frac{(N_j'(h))^2  }{N_j(h)} -  N_j''(h) - \frac{1}{h}  N_j' (h)  \geq 0 
\label{3c}
\eea

We write $T_j := N_j \circ h$; then 
\beqn
W_j (z) = - 2 \log T_j(r) ,  
\eeqn
and the plurisubharmonicity \eqref{1n} follows, if we show that 
$- \frac{d^2}{  dr^2} \log T_j(r) \geq 0$,   for all $r \in \bbR^+$ , 
$ j \in \bbN$. (The composite of a convex, increasing function and a plurisubharmonic  function
is plurisubharmonic, see, for example, \cite[ Prop. 2.2.6.]{Kr})
We have by \eqref{3c}
\eqn
& & -  \frac{d^2 \log T_j}{dr^2} = 
\frac{ \big( T_j'  \big)^2 }{(T_j)^2} - \frac{T_j''}{T_j}
\roweq
 \frac{1}{ N_j } \bigg( 
\frac{( N_j'  )^2 (h')^2 }{ N_j  } 
- (N_j'' ) (h')^2
 -    (N_j' )  h'' \bigg)
 \roweq
 \frac{1}{ N_j } \bigg(  (h')^2
\Big(  
\frac{( N_j'  )^2  }{ N_j } 
- N_j'' 
\Big) 
 -    (N_j' )  h'' \bigg)
\rowgeq
\frac{1}{ N_j } \Big( \frac{  (h')^2 }{h} 
-   h'' \Big)  N_j' 
\eneqn
Since $N_j$ is always positive and increasing, the required positivity of  follows 
from~\eqref{h condition}.
\end{proof}

We will now deduce Theorem~\ref{compact case} from Proposition~\ref{cohomology vanishing}. We will proceed in complete analogy with~\cite[Section 7.2]{H} with the simplification that we work with compact blocks. As was mentioned before, this assumption is not necessary but it suffices for us. The argument in H\"ormander can be repeated word for word. 
Thus, we explain the argument only briefly. We will argue by induction on the dimension of $Q$. In the case when the dimension of $Q$ is zero there is nothing to prove. Let us consider $Q$ of a particular dimension $d$ and let us fiber $Q$ as follows
\beqn
\pi: Q \to [a,b]
\eeqn
The fibers of Q are now compact blocks of dimension $d-1$. Each fiber $\pi^{-1}(e)$ thus satisfies Cartan A. Let us now fix a coherent sheaf $\cF$ in the neighborhood of $Q$ . Thus, there is an open neighborhood $U_e$ of $\pi^{-1}(e)$ and a surjection $(\cA^{\wcheck h}_{U_e})^{\oplus p_e} \to \cF|_{U_e}$. We can choose the neighborhood $U_e$ to be rectangle $Q_e\times [a_e,b_e]$. We now choose a finite sub cover $\{U_i\}_{i\in I}$ of $U_e$ shrinking the $U_e$, if necessary, so that the $U_i$ have the form
\beqn
U_i = \tilde Q \times (a_i,b_i) \ \ \text{with}\ \ \ a_1 < a_2 < b_1 < a_3 < b_2 <\dots a_n <b_{n-1}< b_n\,.
\eeqn
In other words only the consecutive $U_i$ intersect each other. On each overlap $U_i\cap U_{i+1}$ we choose $\alpha_i$ and $\beta_i$ making the diagram below commute:
\beqn
\begin{gathered}
\xymatrix@R=.5pc{(\cA^{\wcheck h}_{U_i\cap U_{i+1}})^{\oplus p_i}\ar@<-.5ex>[dd]_{\alpha_i}\ar[dr]&\\&\cF|_{U_i\cap U_{i+1}}\\(\cA^{\wcheck h}_{U_i\cap U_{i+1}})^{\oplus p_{i+1}}\ar@<-.5ex>[uu]_{\beta_i}\ar[ur]&}
\end{gathered}\,.
\eeqn
We now use the  $\alpha_i$ and $\beta_i$ to construct transition functions, as in~\cite[Section 7.2]{H}. As there are no triple intersections of the $U_i$ the transition functions yield a vector bundle $E$ on $U = \cup U_i$ and a surjection $E\to \cF|_U$. By writing $\cK_1$ for the kernel of the map $E\to \cF|_U$ we obtain an exact sequence 
\begin{equation*}
0 \to \cK_1 \to (\cA^{\wcheck h}_U)^p \to \cF \to 0
\end{equation*}
Thus, making use of proposition~\ref{cohomology vanishing} for $i\geq 1$ we get that 
\begin{equation*}
\oh^i(Q, \cF) \cong \oh^{i+1}(Q,\cK_1)  \,.
\end{equation*}
Repeating this argument for the coherent sheaf $\cK_j$ and always writing $\cK_{j+1}$ for the kernel we get:
\begin{equation*}
\oh^i(Q,\cF) \cong \oh^{i+1}(Q,\cK_1) \cong \dots \cong \oh^{i+p}(Q,\cK_p) \,.
\end{equation*}
When $p\geq 2N$ the right hand side vanishes and Cartan B follows for $Q$. It remains to prove Cartan A for $Q$. 

Let $x\in Q$ and write $i_x: \{x\} \to Q$. Then we have the following exact sequence
\beqn
0 \to \cK_x \to \cF \to (i_x)_*(i_x)^*\cF \to 0\,.
\eeqn
Taking the long exact sequence and using Cartan B we obtain a surjection 
\beqn
 \cF(Q) \to (i_x)_*(i_x)^*\cF=\cF_x/\fm_x\cF_x\,,
\eeqn
where $\fm_x$ is the maximal ideal in $\cA^{\widecheck h}_{X,x}=\cA^{\widecheck {h(x)}}_{X,x}$. As  $(i_x)_*(i_x)^*\cF=\cF_x/\fm_x\cF_x$ is coherent it is finitely generated over $A_{\widecheck {h(x)}}$. This implies that there are finitely many sections of  $\cF(Q)$ which span $(i_x)_*(i_x)^*\cF=\cF_x/\fm_x\cF_x$. By Nakayama's lemma these sections also span  $\cF_x$ over $\cA^{\widecheck h}_{X,x}$. Thus we obtain a map
\beqn
(\cA_Q^{\widecheck h})^{\oplus n_x} \to \cF
\eeqn
which is a surjection on the level of stalks at $x$. Thus, it is a surjection on some open neighborhood $U_x$ of $x$. As $Q$ is compact, it can be covered by a finite number of the $U_x$, say, $U_{x_1}, \dots, U_{x_s}$. Then
\beqn
\oplus_{i=1}^{s} (\cA_Q^{\widecheck h})^{\oplus n_{x_i}} \to \cF
\eeqn
 is the required surjection.

\section{Proof of Proposition~\ref{Dolbeault exactness}}
\label{L2 arguments}

We retain the notation of the previous section. In particular, we denote the coordinates in $\bC^N$ by $z=(z_1, \dots, z_N)$ and the Lebesgue measure on $\bC^N$ by $dx$. As before, $Q$ stands for a compact block. 

\subsection{The structure sheaf case}
We begin by construction the $L^2$-version of the $\bar\del$-complex
\beq
\label{Db}
\cA_Q^{\wcheck h,(0,0)}(Q) \xrightarrow{\bar\del} \cA_Q^{\wcheck h,(0,1)}(Q)   \xrightarrow{\bar\del} \dots \xrightarrow{\bar\del} \cA_Q^{\wcheck h,(0,q)}(Q)  \xrightarrow{\bar\del}
\eeq
 in complete analogy with the classical $L^2$-version of the $\bar\del$-complex. We could as easily work with $(p,q)$-forms, but in what follows we will stick to $(0,q)$-forms to simplify the notation. We will construct a complex, but for an $h(r)$, not $\widecheck{h(r)}$, of the following form:
 \beq
 \label{L2 Dolbeault}
L_{(0,0)}^2(Q ;  H_{h(r)}) \xrightarrow{\bar\del} L_{(0,1)}^2(Q ;  H_{h(r)})   \xrightarrow{\bar\del} \dots \xrightarrow{\bar\del} L_{(0,q)}^2(Q ;  H_{h(r)}) \xrightarrow{\bar\del}\,.
\eeq
The space $L_{(0,q)}^2(Q ;  H_{h(r)})$ is an $L^2$-spaces of forms with values in a varying 
Hilbert spaces $H_{h(r)}$; we will soon define these and also give 
the precise meaning of the vector valued $\bar \partial$ operator.

Let us recall the algebra norms (3.7),
$
\Vert \sum_j a_jt^j \Vert_{h(r)} = \sum_j |a_j| \, \Vert t^j \Vert_{h(r)} 
$ on the spaces $A_{h(r)}$. 
We introduce $L^2$-version of the spaces $A_{h(r)}$ and of the norms as follows:
\beq
\label{3}
\begin{gathered}
H_{h(r)} = \Big\{ a=  \sum_j a_jt^j
\, : \, 
\Vert a \Vert_{2, h(r)}^2  
< \infty \Big\} , \ \ \ \ \mbox{where} 
 \\ 
 \Vert a \Vert_{2, h(r)}^2  = \sum_{j=0}^\infty 
|a_j|^2 \Vert t^j \Vert_{h(r)}^2 .
\end{gathered}
\eeq
We define the space $L^2(Q ;  H_{h(r)})=L_{(0,0)}^2(Q ;  H_{h(r)})$ as a completion of the space of smooth functions $\cC^{\infty,{h}}_{Q}$ (note that there is no hat in this formula). First, we set
\beq
\label{4}
\Vert f; L^2(Q; H_{h(r)}) \Vert^2  = \int_Q \Vert f(x)\Vert^2_{{2, h(r)}} dx
\eeq
for $\cC^{\infty,{h}}_{Q}$ and then we complete it to $L^2(Q ;  H_{h(r)})$ with respect to the norm above. 
As we are working on $\bC^N$, the bundles of forms $\Omega^{0,q}$ are naturally trivialized by choosing the $ d\bar z^\alpha$ as a basis. Thus, we can define a norm on $\omega\in\cA_Q^{h,(0,q)}(Q)$ by the formula
 \beq
\Vert \omega; L_{0,q}^2(Q; H_{h(r)}) \Vert^2  = \sum_{\alpha}  \Vert f_{\alpha}; L^2(Q; H_{h(r)}) \Vert^2
\eeq
for $\omega = \sum_{|\alpha|=q} f_{\alpha}   d\bar z^\alpha $. We complete $\cA_Q^{h,(0,q)}(Q)$ into the corresponding $L^2$-space\linebreak $L_{(0,q)}^2(Q; H_{h(r)})$ with respect to the norm above.

To define the $\bar \partial$-operator let us rephrase our discussion in the language of  \cite[Chapter 4]{H}. We write $\omega= \sum \omega_j t^j$ and consider a particular $j$. Recall that by  Lemma~\ref{lem11} $W_j (z) = - 2 \log \Vert t^j \Vert_{h(r)}$ is plurisubharmonic and we can consider $W_j (z)$ as a plurisubharmonic weight. As in \cite{H}, we write $L^2_{(0,q)} (Q , W_j) $ for forms which are square integrable with respect to the measure $e^{-W_j (z) } dx$. 
Then, by~\cite{H}, the operator $\bar \partial$ is a closed, densely defined 
operator on the weighted $L^2$-space  $L^2_{(0,q)} (Q , W_j) $ where differentiation is understood in the generalized sense. Denoting by  $\cD_j \subset L^2_{(0,q)} (Q , W_j) $ the domain of $\bar 
\partial$, we define the 
vector valued  $\bar \partial$-operator in the space $L_{(0,q)}^2(Q; H_{h(r)})$ 
coordinate wise by setting its domain to be
\bea
D(\bar \partial) = \Big\{ u = \sum_j u_jt^j \, : \,
u_j \in \cD_j \ , \ \sum_j \Vert u_j ;  L^2_{(0,q)} (Q , W_j) \Vert^2 < \infty
\Big\} . \label{94}
\eea
It is straightforward to see that $\bar \partial$ becomes a densely defined and 
closed operator. These definitions give us the complex~ \eqref{L2 Dolbeault}.

\BEL
\label{lem77}
Let 
$\omega\in L_{(0,q+1)}^2(Q ;  H_{h(r)})$ with $ \bar \partial \omega= 0$. Then the
problem $\bar \partial u = \omega$ has a solution such that $u$ belongs to $L_{(0,q)}^2(Q ;  H_{h(r)})$, and we have
\beqn
\Vert u; L_{(0,q)}^2(Q; H_{h(r)}) \Vert \leq 
c \Vert \omega; L_{(0,q+1)}^2(Q; H_{h(r)}) \Vert .
\eeqn
\ENL

\begin{proof} We write
$$
\omega \ = \ \sum_j \omega_j t^j 
$$
where the $\omega_j$ are now usual $L^2$-forms of type $(0,q)$ on $Q$. 
We have $\bar \partial \omega_j = 0$ for every $j$. We now solve this $\bar \partial$-problem
in a controlled way. Recall~\eqref{1n} where we defined the plurisubharmonic functions $W_j (z) 
= - 2 \log \Vert t^j \Vert_{h(r)}$.  
Theorem 4.4.2 of [H] implies that for every $j$ the equation
\beqn
\bar \partial u_j = \omega_j   \label{7}
\eeqn
has a solution $u_j  = \sum_{\alpha} u_{\alpha,j}$ with a very specific bound:
\bea
& & \sum_{\alpha} \int_Q | u_{\alpha,j} |^2 e^{-W_j(z)} (1+r^2)^{-2}  dx
\leq  
\sum_{\alpha} \int_Q | f_{\alpha,j} |^2 e^{-W_j(z)}   dx ;
\label{8a}
\eea
here we have written $\omega_j=\sum_{\alpha} f_{\alpha,j}$.

Since the  $Q$ is a compact block, the weight $(1+r^2)^{-2}$ is 
bounded from below by a constant $c^{-1} >0$ independent of $j$, and \eqref{8a} 
implies
\beq
\begin{gathered}
\sum_{\alpha} \int_Q | u_{\alpha,j} |^2 e^{-W_j(z)}  dx \leq 
c \sum_{\alpha} \int_Q | u_{\alpha,j} |^2 e^{-W_j(z)} (1+r^2)^{-2}  dx\leq
\\
c \sum_{\alpha} \int_Q | f_{\alpha,j} |^2 e^{-W_j(z)}   dx\,.
\end{gathered}  
\label{8}
\eeq

We set 
\beqn
u= \sum_j u_j t^j \,.
\eeqn
By  \eqref{8} and the definitions of the norms in~\eqref{3} and~\eqref{4},
\eqn
& & \Vert u ; L_{(0,q)}^2(Q;H_{h(r)})\Vert^2
= \sum_{j} \sum_{\alpha} 
\int_Q | u_{\alpha,j} |^2 \Vert t^j \Vert_{h(r)}^2  dx
\roweq 
\sum_{j} \sum_{\alpha} 
\int_Q | u_{\alpha,j} |^2 e^{-W_j(z)}  dx
\leq c\sum_{j} \sum_{\alpha} 
\int_Q | f_{\alpha,j} |^2 e^{-W_j(z)}  dx
\roweq  
c \sum_{j} \sum_{\alpha} 
\int_Q | f_{\alpha,j} |^2 \Vert t^j \Vert_{h(r)}^2  dx
 = 
c \Vert f ; L_{(0,q+1)}(Q;H_{h(r)})\Vert . \label{10}
\eneqn
Thus we have constructed a solution to the $\bar\del$-problem.

\end{proof}

We have now proved the exactness of~\eqref{L2 Dolbeault}. We will use it to prove the required exactness of~\eqref{Db}

Let $\omega\in\cA_Q^{\wcheck h,(0,q)}(Q)$ such that $\bar \partial \omega = 0$. We write, as usual, $\omega = \sum_{|\alpha|=q} f_{\alpha}   d\bar z^\alpha $. 
Because of compactness of $Q$, we have
$$
\cA_Q^{\check h(r)}(Q)= \varinjlim_{m \to \infty} \cA_{Q}^{(1+1/m)h(r)}(Q).  
$$
By compactness of $Q$ and continuity of $\omega$, there exist an $m$
such that  $f_{\alpha}  (z) \in A_{(1+1/m) h(r)}$ for all $\alpha$ and all$z \in Q$. 
Moreover, obviously, for any $(a_j)_{j\in \bbN}$, $k$,
\bea
\sum_j |a_j| \Vert t^j \Vert_k
\geq  \Big( \sum_j |a_j|^2 \Vert t^j \Vert_{k}^2 \Big)^{1/2}\,.
\label{1a}
\eea
Thus,  we  have the continuous embedding $A_{(1+1/m) h(r)}
 \hookrightarrow H_{(1+1/m) h(r)}$

Therefore every  $f_{\alpha}$ is a continuous and thus
bounded function $Q \to H_{(1+1/m) h(r)}$, in particular, an element of 
$L^2(Q;  H_{(1+1/m) h(r)})$. Thus, $\omega \in L_{(0,q+1)}^2(Q ;  H_{(1+1/m) h(r)})$
and by Lemma \ref{lem77} we can find an $u\in L_{(0,q)}^2(Q ;  H_{(1+1/m) h(r)})$ with $\bar \partial u = \omega$.

Writing $u= \sum u_j t^j$ and observing that $\bar \partial$-equation is an elliptic PDE with constant coefficients we conclude that the component functions $u_j$ are smooth functions. 

We have an embedding 
\beq
\label{embedding}
H_{(1+1/m) h(r)} \hookrightarrow A_{(1+1/(m+1)) h(r)}\,.
\eeq
To see this we make use of our controlled nuclearity assumption~\eqref{1g}. It implies, in particular, that there is a constant $K$ such that
$$
 \Vert t^j \Vert_{(1+1/(m+1)) h(r)} \leq K j^{-1} 
 \Vert t^j \Vert_{(1+1/m) h(r)}\qquad \text{for all $j$}\,.
$$
Making use of this inequality and the Cauchy-Schwartz inequality
we conclude that
\eqn
& & \Vert \sum_j a_j t^j \Vert_{(1+1/(m+1)) h(r)} = \sum_j |a_j| \Vert t^j \Vert_{(1+1/(m+1)) h(r)}
\leq  \sum_j |a_j|   K j^{-1} 
\Vert t^j \Vert_{(1+1/m) h(r)}
\rowleq 
 K\Big( \sum_j |a_j|^2 \Vert t^j \Vert_{(1+1/m) h(r)}^2 \Big)^{1/2}.
\eneqn
This inequality gives us the embedding~\eqref{embedding}. Thus, we see that $u(z)\in A_{h_1(r)}$ for all $z\in Q$ with $h_1= (1+1/(m+1)) h$. Finally, the operators $\frac{\partial^\alpha}{\partial x^{\alpha}}$ commute with the operator $\bar\del$. Thus,  $\bar \del\frac{\partial^\alpha u}{\partial x^{\alpha}} = \frac{\partial^\alpha \omega}{\partial x^{\alpha}}$ and so we conclude that $\frac{\partial^\alpha u}{\partial x^{\alpha}}(z)\in A_{h_1(r)}$ for all $z\in Q$. Thus, we see that  $u\in\cA_Q^{\wcheck h,(0,q)}(Q)$ proving the exactness of~\eqref{Db} in degrees $q\geq 1$. 

\subsection{The vector bundle case}
Let us now turn to the case of  a vector bundle E and write $n$ for its rank. The argument in this case will be bit more involved than in the case of the structure sheaf $\cA_Q^{\check h}$. In our case it is easy to write down an inner product on the sections of 
the $H_{h(r)}$-valued sections of $\cE\otimes \Omega^{0,q}$ explicitly rather than choosing a Hermitian metric, so we proceed in this manner. This also makes it easier to reduce arguments to component functions as we did in the first part of this section.

Recall our $\bar\del$-complex
\beq
\label{tm}
\cE_Q^{\wcheck h,(0,0)}(Q) \xrightarrow{\bar\del} \cE_Q^{\wcheck h,(0,1)}(Q)  \xrightarrow{\bar\del} \dots \xrightarrow{\bar\del} \cE_Q^{\wcheck h,(0,q)}(Q) \xrightarrow{\bar\del}\,.
\eeq
Just as in the case of $\cA_Q^{\check h(r)}$ we will construct an analogous $L^2$-complex 
 \beq
 \label{L2 Dolbeault vb}
L_{(0,0)}^2(Q , \cE;  H_{h(r)}) \xrightarrow{\bar\del} L_{(0,1)}^2(Q,\cE ;  H_{h(r)})   \xrightarrow{\bar\del} \dots \xrightarrow{\bar\del} L_{(0,q)}^2(Q,\cE ;  H_{h(r)}) \xrightarrow{\bar\del}\,.
\eeq
We will first prove that the $L^2$-complex is exact and then deduce from that, just as before, the exactness of~\eqref{tm}.

The bundle $E$ is defined in an open neighborhood of $Q$. We choose finite open cover $Q'_i$ of $Q$ by {\it open} blocks such that $E|_{Q'_i}$ is trivial. We now choose compact blocks $Q_i\subset Q'_i$ such that the interiors $\mathring Q_i$ also form an open cover of $Q$. We write $G'_{ij}:Q'_i\cap Q'_j \to GL_n(\cA_{\bC^N}^{\check h}(Q'_i\cap Q'_j))$ for the transition matrices with respect to the cover $Q'_i$. Let us write $G_{ij}=G'_{ij}|_{Q_i\cap Q_j}$. As the $Q_i\cap Q_j$ are compact there is an $h_1>h$ such that $G_{ij}:Q_i\cap Q_j \to GL_n(\cA_{\bC^N}^{h_1}(Q_i\cap Q_j))$, for all $i,j$. Conversely, these transition functions give rise to vector bundle $E_1$ such that $E = \cA_{\bC^N}^{\widecheck h}\otimes_{\cA_{\bC^N}^{h_1}}E_1$ on some neighborhood of $Q$. For example, we can choose $h_1 = (1+1/m)h$ and then, of course, $h_1$ also satisfies condition~\eqref{h condition}. 
We will now work with $E_1$ and the function $h_1$.

We can view $L^2$ sections of $\cE_1$ as a system  of 
$n$-tuples $u_i = (u_{i,k})_{k=1}^n$, $i = 1, \ldots,I$, of 
functions $u_{i,k} \in L^2 (Q_i ;  H_{h_1(r)})$ which satisfy
\bea
u_i = G_{ij} u_j   \ \ \ \mbox{on} \ Q_i \cap Q_j . \label{X1.5}
\eea 
As before we trivialize the bundles of forms by our choice of coordinates on $\bC^N$. Thus we can think of 
  smooth $(0,q)$-forms  of $\cE_1$ as systems of $n$-tuples  $u_i = (u_{i,k})_{k=1}^n$ of $(0,q)$-forms $u_{ik} \in C_{(0,q)}^\infty (Q_i; H_{h_1(r)})$, with
  \beqn
u_{ik} = \sum_{ |\alpha|=q} u_{i,k, \alpha} d\bar z^\alpha ,
\eeqn
where $ u_{i,k, \alpha}$ is a $C^\infty$-function with values in $H_{h(r)}$ and where the $u_i$ satisfy \eqref{X1.5}. 

To define the space $L_{(0,q)}^2(Q,\cE_1 ;  H_{h_1(r)})$ of $L^2$-forms
we consider the Hilbert space
\bea
 \cX(q)  =   \prod_{i=1}^{I} \ \prod_{k=1}^n
L^2_{(0,q)}( Q_i ;  H_{h_1(r)}), \label{X1.8}
\eea
We then set
\beq
\label{X1.11}
\begin{gathered}
L_{(0,q)}^2(Q,\cE_1 ;  H_{h_1(r)}) \ = \ \cH(q)=
\\
\Big\{  u  =  (u_{i})_{i} = (u_{ik})_{i,k}
\in   \cX(q) \mid
u_i = G_{ij} u_j \mbox{ for all } i,j \Big\} \,.
\end{gathered}
\eeq

The norm $\Vert u \Vert_q$ on the  space $\cX(q)$ is associated to the
natural inner product
\bea
(g | v )_{q}  = \sum_{i} (g_{i} | v_{i} )_{q,i} 
\label{X1.12}
\eea
where $g = (g_{ik})_{ik}$, $v=(v_{ik})_{ik}$, and 
$( \cdot  | \cdot )_{q,i}$ 
is the inner product of the Hilbert space $\prod_{k=1}^n 
L_{(0,q)}^2(Q_i ;H_{h_1(r)})$.

We define the $\bar \partial$-operator  on $\cH(q)=L_{(0,q)}^2(Q,\cE_1 ;  H_{h_1(r)})$ as a restriction of the $\bar \partial$-operator on $\cX(q)$ which  we define component wise. Since the coefficients $G_{ij}$ are analytic, we have from \eqref{X1.5}
\bea
\bar \partial u_i = G_{ij} \bar \partial u_j. \label{X1.7}
\eea
and hence we obtain  $\bar \partial :L_{(0,q)}^2(Q,\cE_1 ;  H_{h_1(r)})\to L_{(0,q+1)}^2(Q,\cE_1 ;  H_{h_1(r)})$ and thus our complex
 \beq
 \label{L2 Dolbeault'}
L_{(0,0)}^2(Q , \cE_1;  H_{h_1(r)}) \xrightarrow{\bar\del} L_{(0,1)}^2(Q,\cE_1 ;  H_{h_1(r)})   \xrightarrow{\bar\del} \dots \xrightarrow{\bar\del} L_{(0,q)}^2(Q,\cE_1 ;  H_{h_1(r)}) \xrightarrow{\bar\del}\,.
\eeq
We will next prove the exactness of this complex, i.e., 
\begin{prop}
\label{vb L2 exactness}
Let 
$\omega\in L_{(0,q+1)}^2(Q,\cE_1 ;  H_{h_1(r)})$ with $ \bar \partial \omega= 0$. Then there exists $u\in L_{(0,q)}^2(Q,\cE_1 ;  H_{h_1(r)})$ such that $\bar \partial u = \omega$. 
\end{prop}

Let us pause to argue that this proposition implies the vector bundle part of Proposition~\ref{Dolbeault exactness}. Let $\omega\in\cE_Q^{\wcheck h,(0,q+1)}(Q)$ such that $\bar\del\omega=0$. By compactness of $Q$ there exists an $m$ such that $\omega$ is a smooth $(0,q)$ form of $\cE_1$ with $h_1 = (1+1/m)h$. Thus, arguing just as in the trivial bundle case, $\omega$ can be viewed as an element in $L_{(0,q+1)}^2(Q,\cE_1 ;  H_{h_1(r)})$ and by the above lemma we can produce a $u$ such that  $ \bar \partial \omega= 0$ with all the component functions smooth. Continuing to argue as we did earlier in this section we conclude that $u\in \cE_Q^{\wcheck h,(0,q)}(Q)$. Thus the proof of Proposition~\ref{Dolbeault exactness} is complete once we establish the proposition above.

We will now start preparations for the proof of Proposition~\ref{vb L2 exactness}.
For the purposes of the rest of this section we denote, as in \cite{H},  the $\bar \partial$-operator   on $\cX(q)$  by $T$ and the $\bar \partial$-operator on $\cX(q+1)$ by $S$.  Then
\bea
T : \cX(q)   \cap D(T)
\to \cX(q+1) 
, \label{X1.15}
\eea 
where its domain of definition is 
\beq
\begin{aligned}
D(T) &:= &\Big\{( u_{ik})_k = 
\, : \, u_{ik} \in  L^2_{(0,q)}( Q_i ;  H_{h_1(r)}) \cap D(\bar \partial)
 \\
& &  \big\Vert (\bar \partial u_{ik}) ;  L^2_{(0,q+1)}( Q_i ;  H_{h_1(r)}) \big\Vert
< \infty \ \forall k \Big\} ;  
\end{aligned}
 \label{X1.19}
\eeq
recall that  $D(\bar \partial) \subset  L^2_{(0,q)}( Q_i ;  H_{h_1(r)})$
is defined as in \eqref{94}.
 In the same way we have
\bea
S : \cX(q+1)   \cap D(S)
\to \cX(q+2 ) \,.
 \label{X1.16}
\eea 
 Since the Hilbert spaces here are finite Cartesian products, the
next claim follows from the corresponding properties of the component operators.
\BEL
\label{lem87}
Both operators $T$ and $S$ are densely defined and closed. 
\ENL
As a consequence, the adjoint operator 
$T^* :  \cX(q +1 )   \cap D(T^* ) \to \cX(q ) $ and its domain 
can be defined in the standard way. 

Recall that we have written $\cH(q)$ for  $L_{(0,q)}^2(Q,\cE_1 ;  H_{h_1(r)})$. We write 
 $\tilde T: \cH(q)\cap D(T) \to \cH(q+1)$ for the restriction of $T$ and  
$\tilde S:  \cH(q+1)\cap D(S) \to  \cH(q+2)$ for the restriction of $S$. To analyze the operators $\tilde T$ and $\tilde S$
let us write 
\beqn
 \cX^2(q)  =   \prod_{i<j} \ \prod_{k=1}^n 
L^2_{(0,q)}( Q_i\cap Q_j ;  H_{h_1(r)})\,.
\eeqn
Let us introduce the following notation, $(u_i,u_j)$ stands for an entry in $\cX(q)$ which is non-zero only in positions $i,j$ where the entries are as indicated. We have an exact sequence
\beqn
0\to \cH(q) \to  \cX(q)   \xrightarrow{L}  \cX^2(q)
\eeqn
where $L$ is given by 
\beqn
L(u_i,u_j)|_{Q_{i'}\cap Q_{j'}}= \begin{cases}
    0&  \text{if $(i',j')\neq (i,j)$},\\
   u_j- G_{ji}u_i &  \text{if $(i',j')=(i,j)$}.
\end{cases} 
\eeqn
The expression $u_j- G_{ji}u_i$ is to be interpreted as first restricting $u_i$ and $u_j$ to $Q_i\cap Q_j$. The orthogonal complement $\cH(q)^\perp$ is then given as the image of the adjoint $L^*$. An easy calculation shows that 
\beqn
L^*(u_{ij}) = (-^t\bar G_{ji}u_{ij}, u_{ij})\,.
\eeqn
Were we use same notation as above and we interpret the $u_{ij}$ and the $-^t\bar G_{ji}u_{ij}$ as elements in $L^2_{(0,q)}(  Q_j ;  H_{h_1(r)})$ and $L^2_{(0,q)}( Q_i;  H_{h_1(r)})$ using the inclusions (given by extension by zero) $L^2_{(0,q)}( Q_i\cap Q_j ;  H_{h_1(r)})\subset L^2_{(0,q)}(  Q_j ;  H_{h_1(r)})$ and  $L^2_{(0,q)}( Q_i\cap Q_j ;  H_{h_1(r)})\subset L^2_{(0,q)}(  Q_i ;  H_{h_1(r)})$, respectively. 

We now have
\BEL
\label{lem97}
The spaces $\cH(q)$ and $\cH(q)^\perp \subset \cX(q)$ contain  dense subspaces 
$\cH^\infty(q)$ and $\cH^{\infty,\perp}(q)$, respectively,  consisting of 
$C^\infty$-functions.  
\ENL
\begin{proof} The space $\cH^\infty(q)$  is of course just the space of smooth sections of the bundle $\cE_1$ and as such it is of course dense in $\cH(q)$. 
As was pointed above the $\cH(q)^\perp$ is a finite linear combination of elements of the form $L^*(u_{ij}) = (-^t\bar G_{ji}u_{ij}, u_{ij})$. To construct $\cH^{\infty,\perp}(q)$ we take linear combinations of the $L^*(u_{ij})$ where $u_{ij}$ are smooth functions on $Q$ with support in $Q_i\cap Q_j$. The $L^*(u_{ij})$ are smooth and $\cH^{\infty,\perp}(q)$ is dense by construction. 
\end{proof}
As the operators $\tilde T$ and  $\tilde S$ are clearly closed, the lemma above immediately implies that
\beq
\label{lem92}
\text{The operators $\tilde T$ and  $\tilde S$ are closed and densely defined}\,.
\eeq

{\bf Proof of Proposition~\ref{vb L2 exactness}} In the language we just introduce we are to show that if $\omega$ is in the kernel of $\tilde S$ then it is in the image of $\tilde T$. We make use of the following well-known lemma, see~\cite[Lemma 4.1.1]{H}, for example, 
\BEL
\label{lem88}
Let $\cT$ be a linear, closed operator from a dense subspace of 
$H_1$ into $H_2$, where $H_j$ are  Hilbert spaces, and let $F\subset H_2$
be a closed subspace such that the range of $\cT$ satisfies $R_{\cT} \subset F$. Then, 
$R_\cT =F$ if and only if
$ \Vert g \Vert_{H_2} \leq C \Vert \cT^*g \Vert_{H_1} $
for all $g \in F \cap D(\cT^* )$. 
\ENL

Applying Lemma~\ref{lem77} component-wise implies that the range of the  
operator  $T $ is {\rm ker}$\, S$, thus, by Lemma \ref{lem88},
\beqn
\Vert g \Vert \leq C \Vert T^*g \Vert \label{X1.23}
\eeqn
for all $g \in {\rm ker}\, S \cap D(T^* )$ for some $C$. Thus, in order to prove Proposition~\ref{vb L2 exactness}, it suffices to show that
\beqn
\Vert g \Vert \leq C \Vert \tilde T^*g \Vert \qquad\text{for all $g \in {\rm ker} \, \tilde S \cap D(\tilde T^* )$\ \  for some $C$}\,.
\eeqn
To do so, in view of 
\eqref{X1.23},  it suffices to show that $D(\tilde T^* ) = D(T^* ) \cap  \cH(q+1)$ and that 
$\tilde T^*  = T^*$ on $D(\tilde T^* )$. 
By definition, $D(\tilde T^* )$ consists of functions $g \in  \cH(q+1)$ 
for which there exists a constant $C =C(g)> 0$ such that 
\beqn
|(\tilde Tv | g )_{q+1} | \leq C \Vert v \Vert 
\eeqn
for all $v \in  D(\tilde T) \subset \cH(q) $.
Thus, we are reduced to showing that if $g$ is above then we also have 
\beqn
|(Tv | g )_{q+1} | \leq C' \Vert v \Vert  
\eeqn
for all $v \in  D(T) \subset \cX(q) $ for some $C'$.

Note that, of course, $T$ preserves $\cH(q)$, but it does not preserve $\cH(q)^\perp$. However, as we will show next, it preserves $\cH(q)^\perp$ up to a bounded operator. 
Let us now consider $L^*(u_{ij}) = (-^t\bar G_{ji}u_{ij}, u_{ij})\in \cH(q)^\perp$. Then we see that
\beqn
\begin{gathered}
TL^*(u_{ij}) = (\bar \del( -^t\bar G_{ji}u_{ij}),  \bar \del u_{ij}) = (-^t\bar G_{ji}\bar \del u_{ij}, \bar \del u_{ij})+(( -\bar \del^t\bar G_{ji})u_{ij},0)
\\
= \ L^*T(u_{ij}) +(( -\bar \del\, ^t\!\bar G_{ji})u_{ij}, 0) = L^*T(u_{ij}) + R(u_{ij})\,.
\end{gathered}
\eeqn
where $R: \cX^2(q) \to \cX(q+1)$ is the bounded operator given by $R(u_{ij})= (( -\bar \del\, ^t\!\bar G_{ji})u_{ij}, 0) $. 

Let us now pick $w\in\cH^\infty(q)$ and $w^\perp \in \cH^{\infty,\perp}(q)$. The elements $v$ of the form $v=w+w^\perp$ are dense in $\cX(q)$. Note that we can assume that $w^\perp=L^*u_{ij}$ as it is a linear combination of such expressions. Then,
\beq
\begin{gathered}
 |(T v  |g )_{q+1}|  =   |(T (w+w^\perp) |  g )_{q+1}|  
\leq |(T w |  g )_{q+1}| + |(T w^\perp |  g )_{q+1}|  \leq 
\\
C  \Vert w \Vert  + |(T L^*u_{ij}|  g )_{q+1}| =   C  \Vert w \Vert  + |( L^*T u_{ij}|  g )_{q+1}+( R u_{ij}|  g )_{q+1}| =
\\
C  \Vert w \Vert  + |( R u_{ij}|  g )_{q+1}| \leq C  \Vert w \Vert  
+ \Vert R \Vert \, \Vert u_{ij} \Vert  \,  \Vert g \Vert 
\leq  (C+ \Vert R \Vert  \Vert  \,  \Vert g \Vert )\Vert  v \Vert , 
\end{gathered}
\label{X1.34} 
\eeq
because $w^\perp = (-^t\bar G_{ji}u_{ij}, u_{ij})$ and hence $\Vert u_{ij}\Vert \leq \Vert w^\perp \Vert$.
Thus, the linear map 
$v \mapsto (T v  |g )_{q+1}$  
extends to $D(T)$ as a bounded map and  \eqref{X1.34} holds for all $v \in D(T)$. 
In other words,  $g \in D(T^*)$ and therefore  $D(\tilde T^* ) \subset 
D(T^* ) \cap  \cH(q+1)$. Also, the identity $(v |T^* g )_{q} =
(v |\tilde T^* g )_{q} $ holds for  all $v \in D(\tilde T)$ and 
$g \in D(\tilde T^* )$, and hence $T^* = \tilde T^* $ on $D(\tilde T^*)$.
This completes the  proof of Proposition  \ref{vb L2 exactness}.

\section{Topology on the sheaves and approximation lemmas}
\label{approximation}

In this section we define a topology on global sections of coherent sheaves on compact blocks. We use these topologies to prove approximation lemmas which will be used in the next section.

Let $\cF^{\widecheck h}$ be a coherent $\cA^{\check h}_{\bC^N}$-module where, as usual, $h$ is a function satisfying condition~\eqref{h condition}.
We consider a Stein exhaustion $X_\nu$ of $\bC^N$  by  compact blocks, i.e.,  the $X_\nu$, $\nu=1, \dots$  are compact blocks with $X_\nu \subset X_\mu$ if $\nu <\mu$ and $\bigcup X_\nu = \bC^N$. We might as well choose the exhaustion to be given as
\beqn
X_\nu \ = \ \{z=(z_1, \dots, z_N)\in \bC^N\mid  |\Re z_i| \leq \nu \ \ |\Im z_i| \leq \nu \}\,.
\eeqn

We apply the Cartan A part of  theorem~\ref{compact case}, i.e. the existence of surjections
\beq
\label{fixed surjection}
 (\cA^{\widecheck h}_{X_\nu})^{\oplus p_\nu}  \twoheadrightarrow       \cF^{\wcheck h}_{X_\nu}
\eeq
to the compact blocks $X_\nu$ to obtain presentations
\beq
\label{fixed presentation}
 \cA^{\widecheck h}_{\bC^N}(X_\nu)^{\oplus p_\nu}  \twoheadrightarrow       \cF^{\wcheck h}(X_\nu)\,.
\eeq
We now fix these presentations. 
 Recall that we constructed  $\cF^{\wcheck k}$ in section \S\ref{varying levels} associated to the sheaf $\cF^{\wcheck h}$ and a level $k\leq h$. By base change the presentations \eqref{fixed surjection} then give rise  to analogous presentations 
\beqn
 (\cA^{\widecheck k}_{{X_\nu}})^{\oplus p_\nu}  \twoheadrightarrow       \cF^{\widecheck k}_{X_\nu}\,.
\eeqn
and passing to global sections we have
\beq
\label{fixed presentation for k}
 \cA^{\widecheck k}_{\bC^N}(X_\nu)^{\oplus p_\nu}  \twoheadrightarrow       \cF^{\widecheck k}(X_\nu)\,.
\eeq

Let us   write $k:= h/2$. 
The space  $ \cF^{\wcheck h}({X_\nu})$ is a direct limit of an increasing {\it sequence} of
Banach spaces  $ \big( \cF^{\wcheck h}_{\nu , m} , \Vert \cdot \Vert_{ X_\nu, m} \big) 
\subset  \cF^{\wcheck h}({X_\nu})$, $m \in \bbN$. 
To see this we argue as follows. Since $X_\nu$ is compact, for any continuous function
$h_1 : X_\nu \to \bbR_+$ which satisfies $h_1(x) > h(x)$ for $x \in X_\nu$,
we can find $\varepsilon >0$ such that $h_1 \geq (1+ \varepsilon)h$. Hence, the 
space $\cA^{\wcheck h}_{\bC^N} (X_\nu)$ is a  countable direct limit of
Banach spaces, 
\begin{equation}
\cA^{\wcheck h}_{\bC^N} (X_\nu) = \varinjlim_{h_1 > h} \cA^{h_1}_{\bC^N} (X_\nu) 
= \varinjlim_{m \to \infty} \cA^{(1+ 1/m)h}_{\bC^N} (X_\nu) ,
\label{5.12}
\end{equation}
where the norm of  $\cA^{(1+ 1/m)h}_{\bC^N} (X_\nu)  $ is given by
\begin{equation}
\sup\limits_{X_\nu} \Vert f \Vert_{(1 + \frac1m)h} = 
\sup\limits_{x \in X_\nu} \Vert f (x) \Vert_{(1 + \frac1m)h(x)}  = \sup\limits_{x \in X_\nu}
\sum_{j=0}^\infty |a_j(x)| \Vert t^j \Vert_{(1 + 1/m)h(x)}.
\label{5.13}
\end{equation}
for $ f(x,t) = \sum_{j=0}^\infty a_j(x)t^j $ (cf. \eqref{C norm}). Then, 
$( \cA_{\bC^N}^{\wcheck h}({X_\nu}))^{\oplus p_\nu}$ and thus also 
$ \cF({X_\nu})$ are countable direct limits of Banach spaces,
the latter as a quotient in the presentations we fixed in~\eqref{fixed presentation}
\beqn
( \cA_{\bC^N}^{\wcheck h}({X_\nu}))^{\oplus p_\nu} 
\twoheadrightarrow      \cF({X_\nu}) \, .
\eeqn

Now,  $\cA^{\wcheck k}_{\bC^N} (X_\nu)$ has a bounded set  which contains a 
neighborhood of 0 of $\cA^{\wcheck h}_{\bC^N} (X_\nu)$. Indeed, from \eqref{5.13} it is easy to see  that the bounded set can, for example,  be chosen to be
$\{ f \, : \, \sup_{X_\nu} \Vert f  \Vert_{4/3 k} = 
\sup_{X_\nu} \Vert f  \Vert_{2/3 h} \leq 1 \}$, which contains
the neighborhood 
\begin{equation*}
\Gamma \big( \bigcup_{m \in \bbN}  \{ f \, : \, \sup\limits_{X_\nu} 
\Vert f  \Vert_{(1+1/m)h} \leq 1 \}  \big)\,;
\end{equation*} 
here $\Gamma$ denotes the absolutely convex hull. 

Let us recall that by~\eqref{inclusion} we have an inclusion $\ \cF^{\wcheck h}({X_\nu})
\subset \cF^{\wcheck k}({X_\nu}) $. Furthermore:
\begin{lem}
\label{lem150916a}
Given $\nu$, there is a bounded set  $B(\nu) \subset\cF^{\wcheck k}({X_\nu})$
which  
contains a neighborhood $U(\nu)$ of 0 of $\ \cF^{\wcheck h}({X_\nu})
\subset \cF^{\wcheck k}({X_\nu}) $.
\end{lem}

\begin{proof}
By the remark just above, the space $(\cA_{\bC^N}^{\wcheck k}({X_\nu}))^{\oplus p_{\nu}}$
has a bounded set which is a 
neighborhood of 0 of $ (\cA_X^{\wcheck h}({X_\nu}))^{\oplus p_{\nu}}$. 
The claim follows from  the commutative diagram
\beqn
\begin{CD}
 (\cA_{\bC^N}^{\wcheck h}({X_\nu}))^{\oplus p_{\nu}}@>>>(\cA_{\bC^N}^{\wcheck k}({X_\nu}))^{\oplus p_{\nu}}
 \\
 @VVV @VVV 
 \\
\cF^{\wcheck h}({X_\nu}) @>>> \cF^{\wcheck k}({X_\nu}) ,
\end{CD}
\label{5.16}
\eeqn
where the horizontal  mappings are inclusions and the vertical ones are 
continuous surjections defining the topologies of the spaces on the 
last row. 
\end{proof}

We denote the Minkowski functional of the bounded set $B(\nu)$ by
\begin{equation}
\Vert f \Vert_{B(\nu)} = \inf\{ r > 0  \, : \, f \in rB(\nu) \} . 
\label{5.16j}
\end{equation}
This is well defined for $f \in   \cF^{\wcheck h}({X_\nu})$, by the lemma above. Moreover:
\beq
\label{B(n) convergence}
\begin{gathered}
\text{ convergence with respect to $\Vert \cdot \Vert_{B(\nu)}$ implies}
\\
\text{convergence in the topology of
$ \cF^{\wcheck k}({X_\nu}) $}.
\end{gathered}
\eeq

We will now formulate the second approximation theorem we will make use of in the next section. In this lemma we fix the level $h$ but compare the topologies on various $\cF^{\wcheck h}({X_\nu}) $.

To compare these norms for various $X_\nu$ let us consider $X_\nu$ and $X_{\nu+1}$.  By restricting the presentation
\beqn
 (\cA^{\widecheck h}_{X_{\nu+1}})^{\oplus p_{\nu+1}}  \twoheadrightarrow       \cF^{\wcheck h}_{X_{\nu+1}}
\eeqn
to $X_\nu$ we obtain two presentations of  $\cF^{\wcheck h}|_{X_\nu}$ and we then choose  ${s_{\nu+1,\nu}}$ once and for all so that
\beqn
\begin{CD}
\cA_{X_\nu} ^{\oplus p_{\nu+1}} @>{s_{\nu+1,\nu}}>> \cA_{X_\nu} ^{\oplus p_\nu} 
 \\
 @VVV @VVV
 \\
 \cF^{\wcheck h}|_{X_\nu} @=  \cF^{\wcheck h}|_{X_\nu} 
\end{CD}
\eeqn
commutes. Passing to global sections we obtain:
\beq
\label{comparison2}
\begin{CD}
 \cA_{\bC^N}(X_{\nu+1})^{\oplus p_{\nu+1}}@>>>  \cA_{\bC^N}(X_\nu)^{\oplus p_{\nu+1}} @>{s_{\nu+1,\nu}}>> \cA_{\bC^N}(X_\nu)^{\oplus p_\nu} 
 \\
 @VVV @VVV @VVV
 \\
\cF^{\wcheck h}(X_{\nu+1}) @>>> \cF^{\wcheck h}(X_\nu) @=  \cF^{\wcheck h}(X_\nu)
\end{CD}
\eeq
with surjective columns.

\begin{lem}
\label{lem150916b}
Given $\nu$, $\gamma_\nu' \in 
\cF^{}({X_\nu}) $ 
and a collection of continuous seminorms $\Vert \cdot \Vert_{\nu,n}$
on $\cF^{}({X_n}) $, $n \leq \nu$,  then, for any $\varepsilon >0$,
there exists $\gamma_\nu \in \cF^{\wcheck h}({X_{\nu+1}}) $ 
such that for all $n \leq \nu$ 
we have 
\begin{eqnarray}
\Vert \gamma_\nu |_{X_{n}} - \gamma_\nu'|_{X_{n}} \Vert_{\nu,n} \leq \varepsilon 
. \label{1509E}
\end{eqnarray}
\end{lem} 

\begin{proof} $1^\circ$. We show that the restriction map
$ \cA_{\bC^N}^{\wcheck{h}}({X_{\nu+1}})
\to \cA_{\bC^N}^{\wcheck{h}}({X_{\nu}} )$ has dense image.
Given $g_\nu' 
\in \cA_{\bC^N}^{\wcheck{h}}({X_{\nu}})$ 
we write $g_{\nu}'(x) = \sum_{j=0}^\infty a_{\nu,j}(x)
t^j$, where $a_{\nu,  j}$ are scalar
holomorphic mappings on $X_\nu$.  
We also choose $m$ such that $g'_{\nu} \in \cA_X^{(1+1/ m)h}(X_\nu)$
(see \eqref{5.12}), i.e., we have
$g'_{\nu} (x) \in A_{(1+ 1/ m)h(x)}$ for all $x \in X_\nu$.

Given $\varepsilon$ we now choose $l$ such that (cf. \eqref{5.13}) 
\begin{equation*}
\sup\limits_{x \in X_n}
\Vert \sum_{j=l}^\infty a_{\nu,  j}(x) t^j \Vert_{(1 + 1/ m)h(x)}
 = \sup\limits_{x \in X_n} \sum_{j=l}^\infty |a_{\nu,  j}(x)|
\Vert  t^j \Vert_{(1 + 1/ m)h(x)} 
< \varepsilon /2 
\end{equation*}
for all $n \leq \nu$. Then, we have finitely many scalar holomorphic mappings
$ a_{\nu,  j} : X_\nu \to \bbC$, $j < l$, and using the Runge 
approximation theorem on $X_\nu$, 
we approximate all of them by corresponding polynomials 
$ P_{\nu,  j} : X \to \bbC$, $j < k$ such that
\begin{equation*}
\sup\limits_{x \in X_n}
 \sum_{j=0}^{k-1} \big| P_{\nu,  j}(x) -   a_{\nu,  j}(x) \big| 
 \Vert t^j \Vert_{(1 + 1/ m) h(x)} 
< \varepsilon/2  
\end{equation*}
for all $n \leq \nu$. We define $g_{\nu}   \in\cA_{\bC^N}^{\wcheck{h}}({X_{\nu +1}})$
by 
\begin{equation*}
g_{\nu} (x)  =  \sum_{j=0}^{k-1}  P_{\nu,  j} (x) t^j \ \ , \ \ x \in X_{\nu+1}.
\end{equation*}
It is clear from above that for all $n \leq \nu $
\begin{equation*}
\sup\limits_{x \in X_n}
\Vert g_\nu (x) - g_\nu' (x) \Vert_{(1 + 1/ m)h(x)}
< \varepsilon  . \label{5.19a}
\end{equation*}

$2^\circ$. 
We now consider the following commutative diagram:
\beq
\begin{CD}
\big( \cA_{\bC^N}^{\wcheck{h}}({X_{\nu+1}})\big)^{\oplus p_{\nu+1}}@>>> 
 \big(\cA_{\bC^N}^{\wcheck{h}}({X_{\nu}})\big)^{\oplus p_{\nu+1}}
 \\
 @VVV @VVV
 \\
\cF^{\wcheck h}({X_{\nu+1}}) @>>> \cF^{\wcheck h}({X_\nu})
\end{CD}
\label{5.22}
\eeq
By the above argument, the top mapping has dense image.  Also, the vertical mappings in \eqref{5.22} are 
continuous surjections. Hence, the bottom map necessarily also has dense image. Moreover, the expression 
\begin{equation}
\Vert f \Vert_\nu = \sum_{n=1}^\nu \big\Vert f|_{X_n} \big\Vert_{\nu,n}
\label{5.23}
\end{equation}
is a continuous seminorm on $\cF^{\wcheck h}({X_\nu})$, since the restriction maps are continuous.
Consequently, we can approximate 
any element of $\cF^{\wcheck h}({X_\nu})$ arbitrarily well by an element of 
$\cF^{\wcheck h}({X_{\nu+1}})$ with respect to the seminorm \eqref{5.23}. 
Hence, \eqref{1509E} follows.
\end{proof}

\section{Cartan theorems}
\label{main section}

In this section we prove Cartan's  theorems A and B in our setting. As we mentioned in the introduction, we follow the classical  strategy and pass from the compact case of the theorem to the general case by a Stein exhaustion. However, in our setting this process is not as straightforward as in the classical case as controlling various seminorms is more tricky. To control these norms we are forced to allow the auxiliary base rings $A_{\widecheck h}$ to vary along $X$. 

We state our main theorem here for completeness. 
\begin{thm}
\label{main theorem}
Let $X$ be a Stein manifold and let $\cF$ be a coherent  $\cA_X$-module. Then $\oh^i(X,\cF)=0$ for $i\geq 1$. Furthermore, the sheaf $\cF$ is generated by its global sections. 
\end{thm}

\begin{rmk}
We can slightly generalize the theorem above. It also holds for coherent $\cA^{\wcheck h}_X$-module provided that $h$ satisfies condition~\eqref{h condition}. The proof below goes through in this case with minor adjustments. 
\end{rmk}

As we remarked earlier we can assume that $X=\bC^N$ as any Stein manifold can be embedded in $\bC^N$ for some $N$. We first remark that theorem A follows formally from theorem B in a similar manner as as was explained in section~\ref{compact section} for compact blocks. In our setting the argument there gives a surjection $\oh^0(X,\cF)\otimes_A \cA_X\to \cF$.

We now consider a Stein exhaustion $(X_\nu)_{\nu \in \bbN}$ by compact compact blocks in $\bC^N$, as in the previous section.  By theorem~\ref{compact case} we can conclude that
\begin{equation*}
\text{$\oh^i(X_\nu,\cF)=0$ for $i\geq 1$}\,.
\end{equation*}
By a very general argument, see, for example, \cite[Chapter 4, \S1,Theorem 4]{GrRe2} we conclude that  $\oh^i(X,\cF)=0$ for $i\geq 2$. Thus we are left to deal with the case $i=1$.

To prove the vanishing of $\oh^1(X,\cF)$ 
we fix a countable cover of $\bC^N$ by precompact open Stein domains $U_i$, ${i \in \bN}$. We can and will assume that the cover has the property that for any $X_\nu$ only finitely many of the $U_i$ have a non-trivial intersection with $X_\nu$. We will next choose Stein domains $(V_i)_{i\in \bN}$ such that  $V_i\subset \bar V_i \subset U_i$ and such that the $V_i$ still form a cover of $\bC^N$. O course, the sets $\bar V_i$ are compact. 
We also assume to be given a representative $\al\in \prod \cF(U_i\cap U_j)$ of a class in $\oh^1(X,\cF)$; we fix $\alpha $ for the rest of the proof.
Let us now consider the Chech complexes: 
\begin{equation*}
\begin{CD}
\prod \cF(U_i) @>{\delta^0}>> \prod \cF(U_i\cap U_j) @>{\delta^1}>>\prod \cF(U_i\cap U_j\cap U_k) @>>> \dots
\\
@VVV @VVV @VVV 
\\
\prod \cF(\bar V_i) @>{\delta^0}>> \prod \cF(\bar V_i\cap \bar V_j) @>{\delta^1}>> \prod \cF(\bar V_i \cap \bar V_j\cap\bar V_k) @>>> \dots
\\
@VVV @VVV @VVV 
\\
\prod \cF(V_i) @>{\delta^0}>> \prod \cF(V_i\cap V_j) @>{\delta^1}>> \prod \cF(V_i\cap V_j\cap V_k) @>>> \dots
\end{CD}
\end{equation*}
Now, the first row and the third row both compute the cohomology $\oh^*(X,\cF)$ and the restriction map from the first row to the third row induces the identity map on $\oh^*(X,\cF)$.  
Restricting the cocycle $\al$ to the cover $(V_i)$ we obtain 
$\al_V\in \prod \cF(V_i\cap V_j)$. As $\al_V$ comes from $\prod \cF(\bar V_i \cap \bar V_j)$ we see that each component $\al_V(i,j)\in  \cF^{ {h_{i,j}} } (V_i\cap V_j) $ for some constants $ h_{i,j}  > 0$.

Let us choose  a twice differentiable  function $h: X \to \bbR^+$ (recall that $h$ is a function of the norm $r=|x|$, only) such that it is smaller than the above constants $h_{i,j}$ on the sets $U_i \cap U_j$ and satisfies the condition~\eqref{h condition}. This is possible because for each $X_\nu$ there are only finitely many $V_i$ intersecting $X_\nu$. Let us write $h_\nu$ for the minimum of the $h_{i,j}$ arising from the $V_i$ which intersect $X_\nu$. Thus, the function $h: \bbR^+ \to \bbR^+$ has to satisfy:
\beqn
h(r) < h_\nu \ \ \ \  \text{if $\nu-1 \leq r \leq \nu$}\,.
\eeqn
To obtain an $h$ satisfying this condition one simply chooses $H$ to decrease sufficiently rapidly and then $h$ is given by formula~\eqref{big H}.

We also
denote  $k = h/2$ as in Lemma \ref{lem150916a}. 
As a consequence, we have 
\begin{equation}
\alpha_V \in \prod_{i,j} \cF^{\wcheck h} (V_i \cap V_j) \label{alphaincheckh}
\end{equation}
where $\cF^{\wcheck h}$ is the sheaf associated to $\cF$ constructed in section \S\ref{varying levels}. According to the discussion in \S\ref{varying levels}, see in particular~\eqref{inclusion}, we have canonical inclusions $\cF^{\wcheck h}\subset \cF^{\wcheck k} \subset \cF$.

Because $\oh^1(X_\nu, \cF^{\widecheck h})=0$ we conclude that $\alpha_V$ restricted to $X_\nu$ is trivial,  i.e., there is a 
\begin{eqnarray}
\beta_\nu' \in \prod_{i \in I} \cF^{\wcheck h} (U_i \cap X_\nu)  \label{1509BB}
\end{eqnarray} 
such that 
\begin{eqnarray}
( \delta^0 | X_\nu ) \beta_\nu' = \alpha |_{X_\nu} . \label{1509C}
\end{eqnarray}

We have:
\begin{lem}
\label{lem2.3}
Given  $\alpha \in {\rm Ker} \, \delta^1 \subset
 \prod \cF^{\wcheck h} (V_i \cap V_j) $  as in 
\eqref{alphaincheckh},
there exist 
sequences $(\beta_\nu)_{\nu=1}^\infty$, 
$ \beta_\nu \in \prod_i \cF^{\wcheck h} (V_i \cap X_\nu)$,  and 
$(\delta_\nu)_{\nu=1}^\infty$, $\delta_\nu
\in \prod_i \cF^{\wcheck k} (X_{\nu-1})$ 
with the following properties for all $\nu$:
\smallskip

1) $(\delta^0 |_{ X_\nu} ) \beta_\nu = \alpha |_{ X_\nu} $

\smallskip

2) $\big( \beta_{\nu+1}+  \delta_{\nu+1}\big) \big|_{ X_{\nu-1}} = 
\big( \beta_{\nu}+  \delta_{\nu}\big) \big|_{ X_{\nu-1}} \ \text{in \ $\prod_i \cF^{\wcheck k} (V_i\cap X_{\nu-1}) $}$

\end{lem}

Let us first argue that this lemma implies the main result. 
By property 2) of
the lemma, there exists a section $\beta \in \prod_i \cF^{\wcheck k} (V_i)
\subset \prod_i \cF (V_i)$ such that 
$\beta  \big|_{ X_\nu} = \big( \beta_{\nu+1}+  \delta_{\nu+1}\big) \big|_{ X_\nu}$
for all $\nu$. Property 1) of the lemma then implies 
\begin{eqnarray}
(\delta^0 |_{ X_\nu} ) \beta=(\delta^0 |_{ X_\nu} ) ( \beta_{\nu+1}\big|_{ X_\nu} ) +  
(\delta^0 |_{ X_\nu} )(\delta_{\nu+1} \big|_{ X_\nu} ) = 
\alpha |_{ X_\nu} , \nonumber 
\end{eqnarray}
which gives $\delta^0 (\beta) = \alpha$. As $\alpha$ was an arbitrary 1-cocycle we conclude that $\oh^1(X,\cF)=0$. 
Thus, it remains to prove the lemma. 

\vspace{6pt}

{\it Proof of lemma~\ref{lem2.3}}. We make the following definition by induction: assume
that $\nu \in \bbN$ and that 
\begin{eqnarray}
\beta_m 
\in \prod_i \cF^{\wcheck h} (U_i \cap X_m) \ \ , \ \ m \leq \nu
 \label{1509F}
\end{eqnarray} 
have been  chosen such that
\begin{eqnarray}
(\delta^0 |_{ X_m} ) \beta_m = \alpha |_{ X_m} \ \ \mbox{for all} \ m=1, \ldots , \nu .
\label{2.7min}
\end{eqnarray}
We define 
\begin{eqnarray}
\gamma_\nu' = \beta_{\nu +1}' |_{ X_\nu} - \beta_{\nu} , \label{2.7}
\end{eqnarray}
hence,  we have  
\begin{eqnarray}
(\delta^0 |_{ X_\nu} ) \gamma_\nu' =(\delta^0 |_{ X_\nu} ) \big( \beta'_{\nu+1} |_{ X_\nu} 
\big)- 
(\delta^0 |_{ X_\nu} ) \beta_\nu = \alpha |_{X_\nu}- \alpha |_{X_\nu} =
 0  \nonumber
\end{eqnarray}
by \eqref{1509C} and \eqref{2.7min}; as a consequence, $ \gamma_\nu'
\in \cF^{\wcheck h} (X_\nu)$.

We now apply Lemma~\ref{lem150916a} for all $n \leq \nu$. Thus we obtain bounded sets $B(n)\subset \cF^{\wcheck k}(X_n)$ which contain  open neighborhoods $U(n)$ of the origin in  $\cF^{\wcheck h}(X_n)$. We also write  $\Vert \cdot \Vert_{\nu,n}$ for
the continuous seminorms on $\cF^{\wcheck h} (X_n)$ whose unit ball is the open neighborhood $U(n)$.
By Lemma~\ref{lem150916a} and the notation introduced in~\eqref{5.16j}  we can find constants $K(\nu,n) > 1 $ such that
\beq
\Vert f \Vert_{B(n) } \leq K(\nu ,  n ) \Vert f \Vert_{\nu, n} 
\label{1509num}
\eeq
for all $f \in \cF^{\wcheck h} ({X_n})$.
We now apply Lemma \ref{lem150916b} to $\gamma_\nu'$ and thus find $\gamma_\nu
\in \cF^{\wcheck h} ({X_{\nu +1}})$
such that 
\beq
\label{2.11}
\Vert \gamma_\nu' |_{X_n}  - \gamma_\nu  |_{X_n} \Vert_{\nu, n}
\leq  2^{-\nu} \big( \max\limits_{n \leq \nu}  K(\nu, n ) \big)^{-1} .
\eeq
To complete the induction step we define 
\begin{eqnarray}
\beta_{\nu +1} = \beta_{\nu +1 }' - \gamma_\nu 
\in 
\prod_i \cF^{\wcheck h} (U_i \cap X_{\nu+1 } )  , \label{2.13}
\end{eqnarray}
which implies 
\begin{eqnarray*}
(\delta^0 |_{ X_{\nu+1}} ) \beta_{\nu+1} = (\delta^0 |_{ X_{\nu+1} } ) \beta_{\nu+1}' - 
(\delta^0 |_{ X_{\nu+1} } ) \gamma_{\nu} = \alpha  |_{ X_{\nu+1}}  - 0 .
\end{eqnarray*}

It remains to construct the $\delta_\nu$. 
For all $\nu \in \bbN$ we now define 
\begin{eqnarray}
s_j^{(\nu)} = \beta_{\nu + j } |_{X_\nu} - \beta_\nu 
\ ,  \ j = 1,2, \ldots . \label{2.17}
\end{eqnarray}

We claim that the sequence $\big(s_j^{(\nu)} | X_{\nu} 
\big)_{j=1}^\infty$, viewed as a sequence   
in $ \cF^{\wcheck k} ({X_\nu} ) $ converges with respect to 
the seminorm $\Vert \cdot \Vert_{B(\nu)}$,  when $j \to\infty$. Thus,  by~\eqref{B(n) convergence}, the sequence 
$\big(s_j^{(\nu)} | X_{\nu} \big)_{j=1}^\infty$ converges  
in $ \cF^{\wcheck k} ({X_\nu} ) $.

We will then set
\begin{eqnarray}
\delta_\nu =  \lim\limits_{j \to \infty} s_j^{(\nu)}. \label{2.17a}
\end{eqnarray} 

To verify the claim, let $\nu$ be fixed.  We first observe that by \eqref{2.7},
\eqref{2.13}, 
$$
\beta_{n+1} |_{ X_n} - \beta_n = \gamma_n' - \gamma_n|_{X_n} , 
$$
for all $n \leq \nu$, hence, 
\begin{eqnarray*}
 s_j^{(\nu)} 
&=&
\beta_{ j + \nu}|_{X_\nu}   - \beta_{ j + \nu -1 }|_{X_\nu}   
+ \beta_{ j + \nu - 1}|_{X_\nu}    - \ldots  +
\beta_{ \nu +1}|_{X_{\nu+1}} - \beta_{ \nu}|_{X_\nu}  
\nonumber \\ &=&
\sum\limits_{k=0}^{j} \gamma_{ \nu +k  -1}'|_{X_\nu}  - \gamma_{\nu + k -1} |_{X_\nu}
\in \cF^{\wcheck h} (X_\nu) .
\end{eqnarray*}
In the same way,
\begin{eqnarray*}
& & s_j^{(\nu)}   - s_l^{(\nu)}  
\nonumber \\ &=&
\beta_{ j + \nu}|_{X_\nu}   - \beta_{ j + \nu -1 }|_{X_\nu}   
+ \beta_{ j + \nu - 1}|_{X_\nu}    - \ldots  -
\beta_{ l + \nu}|_{X_\nu}   
\nonumber \\ &=&
\sum\limits_{k=l}^{j-1} \gamma_{k + \nu}'|_{X_\nu}  - \gamma_{k+\nu} |_{X_\nu} ,
\end{eqnarray*}
hence, by \eqref{1509num}, \eqref{2.11}, 
\begin{eqnarray*}
& & 
\Vert s_j^{(\nu)}   -  s_l^{(\nu)} \Vert_{B(\nu)} 
\nonumber \\
&\leq &
\sum\limits_{k=l}^{j-1}
K(k + \nu ,\nu) 
\big\Vert \gamma_{k + \nu}'|_{X_\nu}  -  \gamma_{k + \nu} |_{X_\nu} 
\big\Vert_{k+ \nu , \nu }
\nonumber \\
&\leq & \sum\limits_{k=l}^{j-1} 2^{-\nu -k} < \varepsilon ,
\end{eqnarray*} 
if $l $ is large enough. Therefore the $s_j^{(\nu)}$ form a Cauchy sequence in $\cF^{\wcheck k} (X_\nu)$ and so we have
constructed the $\delta_\nu\in \cF^{\wcheck k} (X_\nu)$. 

Now, from \eqref{2.17}, \eqref{2.17a} we deduce that  the sequences
$\big( \beta_{\nu + j } |_{X_\nu} \big)_{j \in \bbN}$ also converge in $\cF^{\wcheck k} (X_\nu)$ and 
hence,
\begin{eqnarray}
&& ( \delta_\nu - \delta_{\nu +1})|_{X_{\nu-1}} 
= \lim\limits_{j \to \infty} (s_j^{(\nu)} - s_j^{(\nu+1)})|_{X_{\nu-1}} 
\nonumber \\ &=& 
- \beta_\nu |_{X_{\nu-1}}   + \beta_{ \nu +1}  |_{X_{\nu-1}} 
 + \lim\limits_{j \to \infty}
\beta_{\nu + j } |_{X_{\nu-1}}  
- \lim\limits_{j \to \infty} \beta_{\nu  +1 + j } |_{X_{\nu-1}}
\nonumber \\ &=& 
- \beta_\nu |_{X_{\nu-1}}   + \beta_{ \nu +1}  |_{X_{\nu-1}} . 
\nonumber
\end{eqnarray}

\end{document}